\documentclass[11pt]{article}

\usepackage[left=2cm,right=2cm,top=1.5cm,bottom=1.5cm]{geometry}
\usepackage{amsmath,amsthm,amssymb}

\usepackage[T1]{fontenc} 
\usepackage[utf8]{inputenc} 
\usepackage{lmodern} 
\usepackage{graphicx}
\usepackage{float}
\usepackage{caption}
\usepackage{subcaption}
\graphicspath{ {images/} }
\usepackage[colorlinks=true, linkcolor=blue, citecolor=blue]{hyperref} 
\usepackage[shortlabels]{enumitem}
\usepackage{tcolorbox}
\usepackage{verbatim}
\usepackage{array}
\usepackage{xcolor}
\usepackage{authblk}
 
\newcommand{\N}{\mathbb{N}}

\newcommand{\R}{\mathbb{R}}

\newtheorem{theorem}{Theorem}
\numberwithin{theorem}{section}

\newtheorem{lemma}[theorem]{Lemma}

\numberwithin{fact}{subsection}

\newtheorem*{theorem*}{Theorem}
\newtheorem*{lemma*}{Lemma}
\newtheorem*{corollary*}{Corollary}

\makeatletter
\newcommand{\settheoremtag}[1]{
  \let\oldthetheorem\thetheorem
  \renewcommand{\thetheorem}{#1}
  \g@addto@macro\endtheorem{
    \addtocounter{theorem}{-1}
    \global\let\thetheorem\oldthetheorem}
  }
\makeatother

\theoremstyle{definition}
\newtheorem{definition}[theorem]{Definition}

\newtheorem{remark}[theorem]{Remark}

\DeclareMathOperator{\zer}{zer}
\DeclareMathOperator{\prox}{prox}
\DeclareMathOperator{\id}{Id}

\DeclareMathOperator{\ran}{ran}

\DeclareMathOperator{\argmin}{argmin}
\DeclareMathOperator{\gra}{gra}

\title{Second order splitting dynamics with vanishing damping for additively structured monotone inclusions \thanks{Research partially supported by the Doctoral Programme \textit{Vienna Graduate School on Computational Optimization (VGSCO)} which is funded by FWF (Austrian Science Fund), project W1260-N35.}}
\author[1]{Radu Ioan Boţ}
\author[2]{David Alexander Hulett}
\affil[1, 2]{Faculty of Mathematics, University of Vienna}
\affil[1]{\href{mailto:radu.bot@univie.ac.at}{radu.bot@univie.ac.at}}
\affil[2]{\href{mailto:david.alexander.hulett@univie.ac.at}{david.alexander.hulett@univie.ac.at}}
\date{}
\begin{document}
 
\maketitle

\begin{abstract}
In the framework of a real Hilbert space, we address the problem of finding the zeros of the sum of a maximally monotone operator $A$ and a cocoercive operator $B$. We study the asymptotic behaviour of the trajectories generated by a second order equation with vanishing damping, attached to this problem, and governed by a time-dependent forward-backward-type operator. This is a splitting system, as it only requires forward evaluations of $B$ and backward evaluations of $A$.  A proper tuning of the system parameters ensures the weak convergence of the trajectories to the set of zeros of $A + B$, as well as fast convergence of the velocities towards zero.  A particular case of our system allows to derive fast convergence rates for the problem of minimizing the sum of a proper, convex and lower semicontinuous function and a smooth and convex function with Lipschitz continuous gradient. We illustrate the theoretical outcomes by numerical experiments. 
\end{abstract}
\bigskip

\medskip

 
\section{Introduction}

\subsection{Problem formulation and a continuous time splitting scheme with vanishing damping}
Let $\mathcal{H}$ be a real Hilbert, $A: \mathcal{H}\to 2^{\mathcal{H}}$ a maximally monotone operator and $B : \mathcal{H}\to \mathcal{H}$ a $\beta$-cocoercive operator for some $\beta > 0$ such that $\zer(A + B)\neq \emptyset$. Devising fast convergent continuous and discrete time dynamics for solving monotone inclusions of the type
\begin{equation}\label{monotoneinclusion}
\mbox{find} \ x \in \mathcal{H} \ \mbox{such that} \ 0 \in (A + B)(x)
\end{equation}
is of great importance in many fields, including, but not limited to, optimization, equilibrium theory, economics and game theory, partial differential equations, and statistics. One of our main motivations comes from the fact that solving the convex optimization problem 
$$\min_{x\in \mathcal{H}} f(x) + g(x),$$
where $f : \mathcal{H} \to \R\cup \{+\infty\}$ is proper, convex and lower semicontinuous and $g : \mathcal{H} \to \R$ is convex and Fréchet differentiable with a Lipschitz continuous gradient, is equivalent to solving the monotone inclusion
\[
0 \in (\partial f + \nabla g)(x).
\]
We want to exploit the additive structure of $(\ref{monotoneinclusion})$ and approach $A$ and $B$ separately, in the spirit of the splitting paradigm.

For $t \geq t_{0} > 0$, $\alpha > 1, \xi \geq 0$, and functions $\lambda, \gamma : [t_{0}, +\infty) \to (0, +\infty)$, we will study the asymptotic behaviour of the trajectories of the second order differential equation
\begin{equation}\label{mainsystem}
\text{(Split-DIN-AVD)}\quad\ddot{x}(t) + \frac{\alpha}{t}\dot{x}(t) + \xi \left(\frac{d}{dt}T_{\lambda(t), \gamma(t)}(x(t))\right) + T_{\lambda(t), \gamma(t)}(x(t)) = 0,
\end{equation}
where, for $\lambda, \gamma > 0$,  the operator $T_{\lambda, \gamma} : \mathcal{H} \to \mathcal{H}$ is given by
\[
T_{\lambda, \gamma} = \frac{1}{\lambda}\Big[\id - J_{\gamma A}\circ(\id - \gamma B)\Big].
\]
The sets of zeros of $A+B$ and of $T_{\lambda, \gamma}$, for $\lambda, \gamma > 0$,  coincide. The nomenclature (Split-DIN-AVD) comes from the splitting feature of the continuous time scheme, as well as the link with the (DIN-AVD) system developed by Attouch and L\'aszl\'o in \cite{AttouchLaszlo} (Dynamic Inertial Newton - Asymptotic Vanishing Damping), which we will emphasize later.  We will discuss the existence and uniqueness of the trajectories generated (Split-DIN-AVD), and also show their weak convergence to the set of zeros of $A + B$ as well as the fast convergence of the velocities to zero, and convergence rates for $T_{\lambda(t), \gamma(t)}(x(t))$ and $\frac{d}{dt}T_{\lambda(t), \gamma(t)}(x(t))$ as $t\to +\infty$.

For the particular case $B = 0$,  we are left with the monotone inclusion problem
$$\mbox{find} \ x \in \mathcal{H} \ \mbox{such that} \ 0 \in A(x),$$
and the attached system
\[
\ddot{x}(t) + \frac{\alpha}{t}\dot{x}(t) + \xi \left(\frac{d}{dt}A_{\lambda(t), \gamma(t)}(x(t))\right) + A_{\lambda(t), \gamma(t)}(x(t)) = 0, 
\]
where, for $\lambda, \gamma > 0$, the operator $A_{\lambda, \gamma} : \mathcal{H} \to \mathcal{H}$ can be seen as a \textit{generalized Moreau envelope} of the operator $A$, i.e.,
\[
A_{\lambda, \gamma} = \frac{1}{\lambda}\Big[\id - J_{\gamma A}\Big].
\]
In particular, we will be able to set $\gamma(t) = \lambda(t)$ for every $t \geq t_0$.  Since for $\lambda > 0$, $A_{\lambda, \lambda} = A_{\lambda}$, this allows us to recover the (DIN-AVD) system
\[
\text{(DIN-AVD)} \quad \ddot{x}(t) + \frac{\alpha}{t}\dot{x}(t) + \xi\left(\frac{d}{dt}A_{\lambda(t)}(x(t))\right) + A_{\lambda(t)}(x(t)) = 0,
\]
addressed by Attouch and László in \cite{AttouchLaszlo}. 

If $A = 0$, and after properly redefining some parameters, we obtain the following system
\[
\ddot{x}(t) + \frac{\alpha}{t}\dot{x}(t) + \xi\left(\frac{d}{dt}\frac{1}{\eta(t)}Bx(t)\right) + \frac{1}{\eta(t)}Bx(t) = 0,
\]
with $\eta : [t_{0}, +\infty) \to (0, +\infty)$,  which addresses the monotone equation
$$\mbox{find} \ x \in \mathcal{H} \ \mbox{such that} \ B(x)=0.$$
This dynamical system approaches the cocoercive operator $B$ directly through a forward evaluation,  which is more natural, instead of having to resort to its Moreau envelope,  as in (DIN-AVD).

\subsection{Notation and preliminaries}

In this subsection, we will explain the notions which were mentioned in the previous subsection, and we will introduce some definitions and preliminary results that will be required later. Throughout the paper, we will be working in a real Hilbert space $\mathcal{H}$ with inner product $\langle \cdot, \cdot\rangle$ and corresponding norm $\| \cdot \| = \sqrt{\langle \cdot, \cdot \rangle}$. 

Let $A : \mathcal{H} \to 2^{\mathcal{H}}$ be a \textit{set-valued} operator, that is, $Ax$ is a subset of $\mathcal{H}$ for every $x\in\mathcal{H}$. The operator $A$ is totally characterized by its \textit{graph} $\gra A = \{(x, u) \in \mathcal{H}\times \mathcal{H} : u\in Ax\}$. The \textit{inverse} of $A$ is the operator $A^{-1} : \mathcal{H} \to 2^{\mathcal{H}}$ well-defined through the equivalence $x\in A^{-1}u$ if and only if $u\in Ax$. The \textit{set of zeros} of $A$ is the set $\zer A = \{x\in \mathcal{H} : 0 \in Ax\}$. For a subset $C\subseteq \mathcal{H}$, we say that $A(C) = \cup_{x\in C}Ax$. The \textit{range} of $A$ is the set $\ran A = A(\mathcal{H})$.

A set-valued operator $A$ is said to be \textit{monotone} if $\langle v - u, y - x\rangle\geq 0$ whenever $(x, u), (y, v)\in\gra A$, and \textit{maximally monotone} if it is monotone and the following implication holds:
    \[
    \widetilde{A} \:\:\text{is monotone}, \:\:\gra A \subseteq \gra\widetilde{A} \:\Longrightarrow\: A = \widetilde{A}.
    \]
Let $\lambda > 0$. The \textit{resolvent} of index $\lambda$ of $A$ is the operator $J_{\lambda A} : \mathcal{H} \to 2^{\mathcal{H}}$ given by 
    \[
    J_{\lambda A} = (\id + \lambda A)^{-1},
    \]
and the \textit{Moreau envelope} (or \textit{Yosida approximation} or \textit{Yosida regularization}) of index $\lambda$ of $A$ is the operator $A_{\lambda} : \mathcal{H} \to 2^{\mathcal{H}}$ given by
    \[
    A_{\lambda} = \frac{1}{\lambda}(\id - J_{\lambda A}),
    \]
where $\id : \mathcal{H} \to \mathcal{H}$, defined by $\id(x) = x$ for every $x\in\mathcal{H}$, is the \textit{identity} operator of  $\mathcal{H}$.   For $\lambda_{1}, \lambda_{2} > 0$, it holds $(A_{\lambda_{1}})_{\lambda_{2}} = A_{\lambda_{1} + \lambda_{2}}$.

A single-valued operator $B : \mathcal{H} \to \mathcal{H}$ is said to be $\beta$-\textit{cocoercive} for some $\beta >0$ if for every $x, y\in \mathcal{H}$ we have
    \[
    \beta \|Bx - By\|^{2} \leq \langle Bx - By, x - y\rangle.
    \]
In this case, $B$ is $\frac{1}{\beta}$-\textit{Lipschitz continuous}, namely,  for every $x, y\in \mathcal{H}$ we have 
    \[
    \|Bx - By\| \leq \frac{1}{\beta} \|x - y\|.
    \]
We say $B$ is \textit{nonexpansive} if it is $1$-Lipschitz continuous, and \textit{firmly nonexpansive} if it is $1$-cocoercive. For $\alpha \in (0, 1)$, we say $B$ is $\alpha$-\textit{averaged} if there exists a nonexpansive operator $R :\mathcal{H} \to \mathcal{H}$ such that 
\[
B = (1 - \alpha)\id + \alpha R.
\]

Let $\lambda > 0$ and $A : \mathcal{H} \to 2^{\mathcal{H}}$. According to \textit{Minty's Theorem}, $A$ is maximally monotone if and only if $\ran(\id + \lambda A) = \mathcal{H}$. In this case $J_{\lambda A}$ is single-valued and firmly nonexpansive, $A_{\lambda}$ is single-valued, $\lambda$-cocoercive, and for every $x\in \mathcal{H}$ and every $\lambda_{1}, \lambda_{2} > 0$ we have 
    \[
    \|J_{\lambda_{1}A}(x) - J_{\lambda_{2}A}(x)\| \leq |\lambda_{1} - \lambda_{2}|\|A_{\lambda_{1}}(x)\|.
    \]

Let $B : \mathcal{H} \to \mathcal{H}$ be a single-valued operator. If $B$ is $\alpha$-averaged for some $\alpha \in (0, 1)$, then $\id - B$ is $\frac{1}{2\alpha}$-cocoercive. If $B$ is monotone and continuous, then it is maximally monotone.

The following concepts and results show the strong interplay between the theory of monotone operators and the convex analysis. 

Let $f : \mathcal{H}\to \R\cup\{+\infty\}$ be a proper, convex and lower semicontinuous function. We denote the infimum of $f$ over $\mathcal{H}$ by $\min_{\mathcal{H}}f$ and the set of global minimizers of $f$ by $\argmin_{\mathcal{H}}f$. The \textit{subdifferential} of $f$ is the operator $\partial f : \mathcal{H} \to 2^{\mathcal{H}}$ defined, for every $x\in \mathcal{H}$, by 
    \[
    \partial f(x) = \{x^{*}\in\mathcal{H} : \langle x^{*}, y - x\rangle + f(x) \leq f(y) \:\:\forall y\in\mathcal{H}\}.
    \]
The subdifferential operator of $f$ is maximally monotone and $\overline x \in \zer \partial f $ $\Leftrightarrow$ $\overline x$ is a global minimizer of $f$.

Let $\lambda > 0$. The \textit{proximal operator} of $f$ of index $\lambda$ is the operator $\prox_{\lambda f} : \mathcal{H} \to \mathcal{H}$ defined, for every $x\in\mathcal{H}$, by 
    \[
    \prox_{\lambda f}(x) = J_{\lambda \partial f}(x)=\argmin_{y\in\mathcal{H}}\left[f(y) + \frac{1}{2 \lambda}\|x - y\|^{2}\right],
    \]
which also means that $\prox_{\lambda f}$ is firmly nonexpansive. The \textit{Moreau envelope} of $f$ of index $\lambda$ is the function $f_{\lambda} : \mathcal{H}\to \R$ given, for every $x\in\mathcal{H}$, by 
    \[
    f_{\lambda}(x) = f\left(\prox_{\lambda f}(x)\right) + \frac{1}{2\lambda}\|x - \prox_{\lambda f}(x)\|^{2}.
    \]
The function $f_{\lambda}$ is Fr\'echet differentiable and 
    \[
    \nabla f_{\lambda}(x) = \frac{1}{\lambda}\left(x - \prox_{\lambda f}(x)\right) = (\partial f)_{\lambda}(x) \quad \forall x \in \mathcal{H}.
    \]
Finally, if $f : \mathcal{H}\to \R$ has full domain and is Fr\'echet differentiable with $\frac{1}{\beta}$-Lipschitz continuous gradient, for $\beta >0$, then, according to \textit{Baillon-Haddad's Theorem}, $\nabla f$ is $\beta$-cocoercive.

\subsection{A brief history of inertial systems attached to optimization problems and monotone inclusions}
In the last years there have been many advances in the study of continuous time inertial systems with vanishing damping attached to monotone inclusion problems. We briefly visit them in the following paragraphs. 

\subsubsection{The Heavy Ball Method with friction}
Consider a convex and continuously differentiable function $f : \mathcal{H}\to\R$ with at least one minimizer. The  \textit{heavy ball with friction} system
\begin{equation}\label{HBF}
\text{(HBF)} \qquad \ddot{x}(t) + \mu\dot{x}(t) + \nabla f(x(t)) = 0
\end{equation}
was introduced by \'Alvarez in \cite{Alvarez} as a suitable continuous time scheme to approach the minimization of the function $f$. This system can be seen as the equation of the horizontal position $x(t)$ of an object that moves, under the force of gravity, along the graph of the function $f$, subject to a kinetic friction represented by the term $\mu\dot{x}(t)$ (a nice derivation can be seen in the work done by Attouch-Goudou-Redont in \cite{AttouchGoudouRedont}).  It is known that, if $x$ is a solution of (HBF), then $x$ converges weakly to a minimizer of $f$ and $f(x(t)) - \min_{\mathcal{H}}f = \mathcal{O}\left(\frac{1}{t}\right)$ as $t\to +\infty$. 

In recent times, the question was raised whether the damping coefficient $\mu$ could be chosen to be time-dependent. An important contribution was made by Su-Boyd-Cand\'es (in \cite{SuBoydCandes}) who studied the case of an Asymptotic Vanishing Damping coefficient $\mu(t) = \frac{\alpha}{t}$, namely, 
\begin{equation}\label{AVD}
\text{(AVD)} \qquad \ddot{x}(t) + \frac{\alpha}{t}\dot{x}(t) + \nabla f(x(t)) = 0,
\end{equation}
and proved when $\alpha \geq 3$ the rate of convergence for the functional values $f(x(t)) - \min_{\mathcal{H}}f = O\left(\frac{1}{t^{2}}\right)$ as $t\to +\infty$. This second order system can be seen as a continuous counterpart to Nesterov's accelerated gradient method from \cite{Nesterov}. Weak convergence of the trajectories generated by $\text{(AVD)}$ when $\alpha > 3$ has been shown by Attouch-Chbani-Peypouquet-Redont \cite{AttouchChbaniPeypouquetRedont} and May \cite{May}, with the improved rate of convergence for the functional values $f(x(t)) - \min_{\mathcal{H}}f = o\left(\frac{1}{t^{2}}\right)$ as $t\to +\infty$. For $\alpha = 3$, the convergence of the trajectories remains an open question, except for the one dimensional case (see \cite{AttouchChbaniRiahi}). In the subcritical case $\alpha\leq 3$, it has been shown by Apidopoulos-Aujol-Dossal \cite{Apidopoulos} and Attouch-Chbani-Riahi \cite{AttouchChbaniRiahi} that the objective values converge at a rate $\mathcal{O}(t^{-\frac{2\alpha}{3}})$ as $t\to +\infty$.

\subsubsection{Heavy Ball dynamics and cocoercive operators}

If $f : \mathcal{H}\to \R\cup\{+\infty\}$ is a proper, convex and lower semicontinuous function which is not necessarily differentiable, then we cannot make direct use of $(\ref{HBF})$. However, since for $\lambda > 0$ we have $\argmin f = \argmin f_{\lambda}$, we can replace $f$ by its Moreau envelope $f_{\lambda}$, and the system now becomes 
\[
\ddot{x}(t) + \mu \dot{x}(t) + \nabla f_{\lambda}(x(t)) = 0.
\]
In line with this idea, and in analogy with $(\ref{HBF})$, \'Alvarez and Attouch \cite{AlvarezAttouch} and Attouch and Maing\'e \cite{AttouchMainge} studied the dynamics
\begin{equation}\label{cocoercive sin vanishing damping}
\ddot{x}(t) + \mu\dot{x}(t) + B(x(t)) = 0,
\end{equation}
where $B : \mathcal{H}\to \mathcal{H}$ is a $\beta$-cocoercive operator. They were able to prove that the solutions of this system weakly converge to elements of $\zer B$ provided that the cocoercitivity parameter $\beta$ and the damping coefficient $\mu$ satisfy $\beta \mu^{2} > 1$. For a maximally monotone operator $A : \mathcal{H}\to 2^{\mathcal{H}}$, we know that its Moreau envelope is $\lambda$-cocoercive and thus, under the condition $\lambda\mu^{2} > 1$, the trajectories of 
\[
\ddot{x}(t) + \mu\dot{x}(t) + A_{\lambda}(x(t)) = 0
\]
converge weakly to elements of $\zer A_{\lambda} = \zer A$. 

Also related to $(\ref{cocoercive sin vanishing damping})$, Boţ-Csetnek \cite{BotCsetnek} considered the system
\begin{equation}\label{systemBotCsetnek}
\ddot{x}(t) + \mu(t)\dot{x}(t) + \nu(t)Bx(t) = 0,
\end{equation}
where $B : \mathcal{H}\to \mathcal{H}$ is again $\beta$-cocoercive. Under the assumption that $\mu$ and $\nu$ are locally absolutely continuous, $\dot{\mu}(t) \leq 0 \leq \dot{\nu}(t)$ for almost every $t\in[0, +\infty)$ and $\inf_{t\geq 0} \frac{\mu^{2}(t)}{\nu(t)} > \frac{1}{\beta}$, the authors were able to prove that the solutions to this system converge weakly to zeros of $B$.

In \cite{AttouchPeypouquet}, Attouch and Peypouquet addressed the system
\begin{equation}\label{DINAVD no damping}
\ddot{x}(t) + \frac{\alpha}{t}\dot{x}(t) + A_{\lambda(t)}(x(t)) = 0,
\end{equation}
where $\alpha > 1$ and the time-dependent regularizing parameter $\lambda(t)$ satisfies $\lambda(t) \frac{\alpha^{2}}{t^{2}} > 1$ for every $t \geq t_0 >0$. As well as ensuring the weak convergence of the trajectories towards elements of $\zer A$, choosing the regularizing parameter in such a fashion allowed the authors to obtain fast convergence of the velocities and accelerations towards zero. 

\subsubsection{Inertial dynamics with Hessian damping}
Let us return briefly to the $\text{(AVD)}$ system $(\ref{AVD})$. In addition to the viscous vanishing damping term $\frac{\alpha}{t}\dot{x}(t)$, the following system with Hessian-driven damping was considered by Attouch-Peypouquet-Redont in \cite{AttouchPeypouquetRedont}
\[
\ddot{x}(t) + \frac{\alpha}{t}\dot{x}(t) + \xi \nabla^{2}f(x(t))\dot{x}(t) + \nabla f(x(t)) = 0,
\]
where $\xi \geq 0$. While preserving the fast convergence properties of the Nesterov accelerated method, the Hessian-driven damping term reduces the oscillatory aspect of the trajectories. In \cite{AttouchLaszlo}, Attouch and L\'aszl\'o studied a version of $(\ref{DINAVD no damping})$ with an added Hessian-driven damping term:
\[
\ddot{x}(t) + \frac{\alpha}{t}\dot{x}(t) + \xi \left(\frac{d}{dt}A_{\lambda(t)}(x(t))\right) + A_{\lambda(t)}(x(t)) = 0.
\]
While preserving the convergence results of $(\ref{DINAVD no damping})$, the main benefit of the introduction of this damping term is the fast convergence rates that can be obtained for $A_{\lambda(t)}(x(t))$ and $\frac{d}{dt}A_{\lambda(t)}(x(t))$ as $t\to +\infty$. The regularizing parameter $\lambda(t)$ is again chosen to be time-dependent; in the general case, the authors take $\lambda(t) = \lambda t^{2}$, and in \cite{AttouchPeypouquet} it is shown that taking $\lambda(t)$ this way is critical. However, in the case where $A = \partial f$ for a proper, convex and lower semicontinuous function $f$, it is also allowed to take $\lambda(t) = \lambda t^{r}$ with $r \geq 0$. 

\subsection{Layout of the paper}
In Section \ref{sec2}, we give the proof for the existence and uniqueness of strong global solutions to (Split-DIN-AVD) by means of a Cauchy-Lipschitz-Picard argument. In Section \ref{sec3} we state the main theorem of this work, and we show the weak convergence of the solutions of $(\ref{mainsystem})$ to elements of $\zer(A + B)$, as well as the fast convergence of the velocities and accelerations to zero. We also provide convergence rates for $T_{\lambda(t), \gamma(t)}(x(t))$ and $\frac{d}{dt}T_{\lambda(t), \gamma(t)}(x(t))$ as $t\to +\infty$. We explore the particular cases $A = 0$ and $B = 0$, and show improvements with respect to previous works. In Section \ref{sec4}, we address the convex minimization case, namely, when $A = \partial f$ and $B = \nabla g$, where $f : \mathcal{H}\to \R\cup \{+\infty\}$ is a proper, convex and lower semicontinuous function and $g : \mathcal{H}\to \R$ is a convex and Fr\'echet differentiable function with Lipschitz continuous gradient, and derive, in addition, a fast convergence rate for the function values. In Section \ref{sec5}, we illustrate the theoretical results by numerical experiments. In Section \ref{sec5}, we provide an algorithm that arises from a time discretization of (Split-DIN-AVD) and discuss its convergence properties.

\section{Existence and uniqueness of trajectories}\label{sec2}
In this section, we show the existence and uniqueness of strong global solutions to (Split-DIN-AVD). For the sake of clarity, first we state the definition of a strong global solution.

\begin{definition}
We say that $x : [t_{0}, +\infty) \to \mathcal{H}$ is a \textit{strong global solution} of (Split-DIN-AVD) with Cauchy data $( x_{0}, u_{0}) \in \mathcal{H} \times  \mathcal{H}$ if
\begin{enumerate}[(i)]
    \item $x, \dot{x} : [t_{0}, +\infty) \to \mathcal{H}$ are locally absolutely continuous;
    \item $\ddot{x}(t) + \frac{\alpha}{t}\dot{x} + \xi\left(\frac{d}{dt}T_{\lambda(t), \gamma(t)}(x(t))\right) + T_{\lambda(t), \gamma(t)}(x(t)) = 0$ for almost every $t\in [t_{0}, +\infty)$;
    \item $x(t_{0}) = x_{0}$, $\dot{x}(t_{0}) = u_{0}$.
\end{enumerate}
A \text{classic solution} is just a strong global solution which is $\mathcal{C}^{2}$. Sometimes we will mention the terms \textit{strong global solution} or \textit{classic global solution} without explicit mention of the Cauchy data.
\end{definition}

The following lemma will be used to prove the existence of strong global solutions of our system, and we will need it in the proof of the main theorem as well. 
\begin{lemma}\label{lemma1}
Let $A : \mathcal{H}\to 2^{\mathcal{H}}$ be a maximally monotone operator and $B : \mathcal{H} \to \mathcal{H}$ a $\beta$-cocoercive operator for some $\beta > 0$. Then, the following statements hold: 
\begin{enumerate}[(i)]
    \item For $\lambda > 0$ and $\gamma\in (0, 2\beta)$, $T_{\lambda, \gamma}$ is a $\lambda\frac{4\beta - \gamma}{4\beta}$-cocoercive operator. In particular, this also implies that $T_{\lambda, \gamma}$ is $\frac{\lambda}{2}$-cocoercive.
    \item Choose $\lambda_{1}, \lambda_{2} > 0$, $\gamma_{1}, \gamma_{2}\in(0, 2\beta)$ and $x, y\in\mathcal{H}$. Then, for $\overline{x}\in\zer(A + B)$ it holds
    \begin{align*}
    \|\lambda_{1}T_{\lambda_{1}, \gamma_{1}}(x) - \lambda_{2}T_{\lambda_{2}, \gamma_{2}}(y)\| \leq \: &4\|x - y\| + \frac{4\beta|\gamma_{1} - \gamma_{2}|}{\gamma_{1}}\|B(x)\| + \frac{2|\gamma_{1} - \gamma_{2}|}{\gamma_{1}}\|x - \overline{x}\|, \\
    \left\|T_{\lambda_{1}, \gamma_{1}}(x) - T_{\lambda_{2}, \gamma_{2}}(y)\right\| \leq \: & \frac{1}{\lambda_{1}}\left[4\|x - y\| + 4\beta\frac{|\gamma_{1} - \gamma_{2}|}{\gamma_{1}}\|Bx\| + 2\frac{|\gamma_{1} - \gamma_{2}|}{\gamma_{1}}\|x - \overline{x}\| \right] \\
    &+  2\frac{|\lambda_{2} - \lambda_{1}|}{\lambda_{1}\lambda_{2}}\|y - \overline{x}\|.
    \end{align*}
    \item If $x$ is a classic global solution to $(\ref{mainsystem})$ and $\overline{x}\in\zer(A + B)$, then, for every $t\geq t_{0}$, we have 
    \[
    \left\|\frac{d}{dt}\left(\lambda(t)T_{\lambda(t), \gamma(t)}(x(t))\right)\right\| \leq 4\|\dot{x}(t)\| + 4\beta\frac{|\dot{\gamma}(t)|}{\gamma(t)}\|B(x(t))\| + 2\frac{|\dot{\gamma}(t)|}{\gamma(t)}\|x(t) - \overline{x}\|.
    \]
\end{enumerate}
\end{lemma}

\begin{proof}
(i) From \cite[Proposition 26.1(iv)(d)]{BC} we know that the operator $J_{\gamma A}\circ (\id - \gamma B)$ is $\alpha = \frac{2\beta}{4\beta - \gamma}$-averaged. From \cite[Proposition 4.39]{BC}, we obtain that $\id - J_{\gamma A}\circ(\id - \gamma B)$ is $\frac{1}{2\alpha}$-cocoercive, namely, it is $\frac{4\beta - \gamma}{4\beta}$-cocoercive. Since $\gamma\in (0, 2\beta)$, we have $\frac{4\beta - \gamma}{4\beta} > \frac{2\beta}{4\beta} = \frac{1}{2}$, which implies that $\id - J_{\gamma A}\circ(\id - \gamma B)$ is $\frac{1}{2}$-cocoercive and thus
    \[
    T_{\lambda, \gamma} \:\:\:\text{is}\:\:\:\lambda\frac{4\beta - \gamma}{4\beta}\text{-cocoercive}\:\:\:\text{and}\:\:\:T_{\lambda, \gamma} \:\:\:\text{is}\:\:\:\frac{\lambda}{2}\text{-cocoercive}.
    \]
    
(ii) We have
\begin{align*}
    \|\lambda_{1}T_{\lambda_{1}, \gamma_{1}}(x) - \lambda_{2}T_{\lambda_{2}, \gamma_{2}}(y)\| \leq  & \ \|x - y\| + \|J_{\gamma_{1}A}(x - \gamma_{1}B(x)) - J_{\gamma_{2}A}(y - \gamma_{2}B(y))\| \\
    \leq & \ \|x - y\| + \|J_{\gamma_{1}A}(x - \gamma_{1}B(x)) - J_{\gamma_{2}A}(x - \gamma_{1}B(x))\| \\
    &+ \|J_{\gamma_{2}A}(x - \gamma_{1}B(x)) - J_{\gamma_{2}A}(y - \gamma_{2}B(y))\| \\
    \leq & \ 2\|x - y\| + |\gamma_{1} - \gamma_{2}|\|A_{\gamma_{1}}(x - \gamma_{1}B(x))\| + \|\gamma_{1}B(x) - \gamma_{2}B(y)\| \\
    \leq & \ 2\|x - y\| + |\gamma_{1} - \gamma_{2}|\|A_{\gamma_{1}}(x - \gamma_{1}B(x))\| \\
    & + \|\gamma_{1}B(x) - \gamma_{2}B(x)\| + \|\gamma_{2}B(x) - \gamma_{2}B(y)\| \\
    = & \ 2\|x - y\| + |\gamma_{1} - \gamma_{2}|\|A_{\gamma_{1}}(x - \gamma_{1}B(x))\| \\
    & \ + |\gamma_{1} - \gamma_{2}|\|B(x)\| + \gamma_{2}\|B(x) - B(y)\|.
\end{align*}
Now, notice that 
\begin{align*}
A_{\gamma_{1}}(x - \gamma_{1}B(x)) &= \frac{1}{\gamma_{1}}(\id - J_{\gamma_{1}A})(x - \gamma_{1}B(x)) = - B(x) + \frac{1}{\gamma_{1}}(x - J_{\gamma_{1}A}(x - \gamma_{1}B(x))) \\
&= -B(x) + T_{\gamma_{1}, \gamma_{1}}(x),
\end{align*}
so using (i) and the fact that $T_{\gamma_{1}, \gamma_{2}}(\overline{x}) = 0$, we obtain
\begin{equation}\label{aux3}
\begin{split}
    \|A_{\gamma_{1}}(x - \gamma_{1}B(x))\| &= \|-B(x) + T_{\gamma_{1},\gamma_{2}}(x)\| \leq \|B(x)\| + \|T_{\gamma_{1}, \gamma_{2}}(x) - T_{\gamma_{1}, \gamma_{2}}(\overline{x})\| \\
    &\leq \|B(x)\| + \frac{2}{\gamma_{1}}\|x - \overline{x}\|.
\end{split}
\end{equation}
Altogether, plugging $(\ref{aux3})$ into our initial inequality yields
\begin{align*}
    \|\lambda_{1}T_{\lambda_{1}, \gamma_{2}}(x) - \lambda_{2}T_{\lambda_{2}, \gamma_{2}}(y)\| \leq \: &2\|x - y\| + 2|\gamma_{1} - \gamma_{2}|\|B(x)\| + \frac{2|\gamma_{1} - \gamma_{2}|}{\gamma_{1}}\|x - \overline{x}\| + \gamma_{2}\|B(x) - B(y)\| \\
    \leq \: &2\|x - y\| + \frac{4\beta|\gamma_{1} - \gamma_{2}|}{\gamma_{1}}\|B(x)\| + \frac{2|\gamma_{1} - \gamma_{2}|}{\gamma_{1}}\|x - \overline{x}\| + 2\beta\left(\frac{1}{\beta}\right)\|x - y\|.
\end{align*}
To show the second inequality, we use the previous one. We have
\begin{align*}
    \left\|T_{\lambda_{1}, \gamma_{1}}(x) - T_{\lambda_{2}, \gamma_{2}}(y)\right\| =   & \ \frac{1}{\lambda_{1}}\left\|\lambda_{1}T_{\lambda_{1}, \gamma_{1}}(x) - \lambda_{2}T_{\lambda_{2}, \gamma_{2}}(y) + (\lambda_{2} - \lambda_{1})T_{\lambda_{2}, \gamma_{2}}(y)\right\| \\
    \leq & \ \frac{1}{\lambda_{1}}\left[4\|x - y\| + 4\beta\frac{|\gamma_{1} - \gamma_{2}|}{\gamma_{1}}\|Bx\| + 2\frac{|\gamma_{1} - \gamma_{2}|}{\gamma_{1}}\|x - \overline{x}\|\right] + \frac{|\lambda_{2} - \lambda_{1}|}{\lambda_{1}}\left\|T_{\lambda_{2}, \gamma_{2}}(y)\right\| \\
    \leq & \ \frac{1}{\lambda_{1}}\left[4\|x - y\| + 4\beta\frac{|\gamma_{1} - \gamma_{2}|}{\gamma_{1}}\|Bx\| + 2\frac{|\gamma_{1} - \gamma_{2}|}{\gamma_{1}}\|x - \overline{x}\|\right] + 2\frac{|\lambda_{2} - \lambda_{1}|}{\lambda_{1}\lambda_{2}}\left\|y - \overline{x}\right\|,
\end{align*}
where the last line is a consequence of $T_{\lambda_{2}, \gamma_{2}}$ being $\frac{\lambda_{2}}{2}$-cocoercive, and hence $\frac{2}{\lambda_{2}}$-Lipschitz continuous (see (i)). 

(iii) For $t, s \geq t_{0}$ set
    \[
    x = x(t),\:\: y = x(s),\:\: \lambda_{1} = \lambda(t),\:\: \gamma_{1} = \gamma(t),\:\: \lambda_{2} = \lambda(s),\:\: \gamma_{2} = \gamma(s)
    \]
and use (ii) to obtain, for every $t\geq t_{0}$,
\begin{align*}
& \ \frac{\|\lambda(t)T_{\lambda(t), \gamma(t)}(x(t)) - \lambda(s)T_{\lambda(s), \gamma(s)}(x(s))\|}{|t - s|} \\
    \leq & \ 4\frac{\|x(t) - x(s)\|}{|t - s|} + \frac{4\beta}{\gamma(t)}\frac{|\gamma(t) - \gamma(s)|}{|t - s|}\|B(x(t))\| +\: \frac{2}{\gamma(t)}\frac{|\gamma(t) - \gamma(s)|}{|t - s|}\|x(t) - \overline{x}\|.  
\end{align*}
Hence, by taking the limit as $s\to t$ we get, for any $t\geq t_{0}$,
\begin{align*}
    \left\|\frac{d}{dt}\lambda(t)T_{\lambda(t), \gamma(t)}(x(t))\right\| \leq 4\|\dot{x}(t)\| + 4\beta\frac{|\dot{\gamma}(t)|}{\gamma(t)}\|B(x(t))\| + 2\frac{|\dot{\gamma}(t)|}{\gamma(t)}\|x(t) - \overline{x}\|.
\end{align*}
\end{proof}

The next theorem concerns the existence and uniqueness of strong global solutions to (Split-DIN-AVD).

\begin{theorem}\label{existence}
Assume that $\lambda, \gamma : [t_{0}, +\infty)\to (0, +\infty)$ are Lebesgue measurable functions and that $\inf_{t\geq t_{0}}\lambda(t) > 0$. Then, for any $(x_{0}, u_{0})\in\mathcal{H}\times\mathcal{H}$ there exists a unique strong global solution $x : [t_{0}, +\infty)\to \mathcal{H}$ of the system $(\ref{mainsystem})$ that satisfies $x(t_{0}) = x_{0}$ and $\dot{x}(t_{0}) = u_{0}$. 
\end{theorem}
\begin{proof}
We will rely on \cite[Proposition 6.2.1]{Haraux} and distinguish between the cases $\xi > 0$ and $\xi = 0$. For each chase, we will check that the conditions of the afforementioned proposition are fulfilled. We will be working in the real Hilbert space $\mathcal{H}\times\mathcal{H}$ endowed with the norm $\|(x, y)\| = \|x\| + \|y\|$. Let $\overline{x}\in\zer(A + B)$ be fixed. 

\textbf{The case} $\xi > 0$. First, it can be easily checked (see also \cite{AlvarezAttouchBolteRedont, AttouchLaszlo, AttouchPeypouquetRedont}) that for all $t\geq t_{0}$ the following dynamical systems are equivalent
    \begin{itemize}
    \item[$*$] $\displaystyle\ddot{x}(t) + \frac{\alpha}{t}\dot{x}(t) + \xi \left(\frac{d}{dt}T_{\lambda(t), \gamma(t)}(x(t))\right) + T_{\lambda(t), \gamma(t)}(x(t)) = 0$.
    \item[$*$] $\displaystyle\begin{cases}
    \dot{x}(t) + \xi T_{\lambda(t), \gamma(t)}(x(t)) - \left(\frac{1}{\xi} - \frac{\alpha}{t}\right)x(t) + \frac{1}{\xi}y(t) = 0, \\
    \dot{y}(t) - \left(\frac{1}{\xi} - \frac{\alpha}{t} + \frac{\alpha\xi}{t^{2}}\right)x(t) + \frac{1}{\xi}y(t) = 0.
    \end{cases}$
    \end{itemize}
In other words, $(\ref{mainsystem})$ with Cauchy data $(x_{0}, u_{0}) = (x(t_{0}), \dot{x}(t_{0}))$ is equivalent to the first order system
\[
\begin{cases}
\dot{z}(t) = F(t, z(t)), \\
z(t_{0}) = (x_{0}, y_{0}),
\end{cases}
\]
where $z(t) = (x(t), y(t))$, $F$ is given, for every $t\geq t_{0}$, by
\[
F(t, (x, y)) = \left[-\xi T_{\lambda(t), \gamma(t)}(x) + \left(\frac{1}{\xi} - \frac{\alpha}{t}\right)x - \frac{1}{\xi}y, \: \left(\frac{1}{\xi} - \frac{\alpha}{t} + \frac{\alpha\xi}{t^{2}}\right)x - \frac{1}{\xi}y\right]
\]
and the Cauchy data is $x_{0} = x(t_{0})$, $y_{0} = -\xi\left(u_{0} + \xi T_{\lambda(t_{0}), \gamma(t_{0})}(x_{0}) - \left(\frac{1}{\xi} - \frac{\alpha}{t_{0}}\right)x_{0}\right)$. 


(i) Let $t\in [t_{0}, +\infty)$ be fixed. We need to verify the Lipschitz continuity of $F$ on the $z$ variable. Set $z = (x, y)$, $w = (u, v)$. We have
\begin{align*}
    \|F(t, z) - F(t, w)\| = \: &\left\|-\xi\left(T_{\lambda(t), \gamma(t)}(x) - T_{\lambda(t), \gamma(t)}(u) + \left(\frac{1}{\xi} - \frac{\alpha}{t}\right)(x - u) - \frac{1}{\xi}(y - v)\right)\right\| \\
    &+ \left\|\left(\frac{1}{\xi} - \frac{\alpha}{t} + \frac{\alpha\xi}{t^{2}}\right)(x - u) - \frac{1}{\xi}(y - v)\right\|.
\end{align*}
Set $\underline{\lambda} := \inf_{t\geq t_{0}}\lambda(t)  > 0$. According to Lemma \ref{lemma1}(i), the term involving the operator $T_{\lambda(t), \gamma(t)}$ satisfies
\[
\left\|T_{\lambda(t), \gamma(t)}(x) - T_{\lambda(t), \gamma(t)}(u)\right\| \leq \frac{2}{\lambda(t)}\|x - u\| \leq \frac{2}{\underline{\lambda}}\|x - u\|.
\]
It follows that, if we take
\[
K(t) := \max\left\{\frac{2\xi}{\underline{\lambda}} + \left|\frac{1}{\xi} - \frac{\alpha}{t}\right| + \left|\frac{1}{\xi} - \frac{\alpha}{t} + \frac{\alpha\xi}{t^{2}}\right|, \: \frac{2}{\xi}\right\} \quad \forall t\geq t_{0},
\]
then we have $K\in L_{\text{loc}}^{1}([t_{0}, +\infty), \R)$ and 
\[
\|F(t, z) - F(t, w)\| \leq K(t)\|z - w\| \quad \forall t\geq t_{0}.
\]

(ii) Now, we claim that $F$ fulfills a boundedness condition. For $t\in [t_{0}, +\infty)$ and $z = (x, y)\in \mathcal{H} \times \mathcal{H}$ we have
\begin{align*}
    \|F(t, z)\| = \left\|-\xi T_{\lambda(t), \gamma(t)}(x) + \left(\frac{1}{\xi} - \frac{\alpha}{t}\right)x - \frac{1}{\xi}y\right\|  + \left\|\left(\frac{1}{\xi} - \frac{\alpha}{t} + \frac{\alpha\xi}{t^{2}}\right)x - \frac{1}{\xi}y\right\|.
\end{align*}
By Lemma \ref{lemma1}(i), we have, for every $t\geq t_{0}$,
\[
\left\|T_{\lambda(t), \gamma(t)}(x)\right\| = \left\|T_{\lambda(t), \gamma(t)}(x) - T_{\lambda(t), \gamma(t)}(\overline{x})\right\| \leq \frac{2}{\lambda(t)}\|x - \overline{x}\|.
\]
Hence, if we take 
\[
P(t) = \max\left\{\frac{2\xi}{\lambda(t)} + \left|\frac{1}{\xi} - \frac{\alpha}{t}\right| + \left|\frac{1}{\xi} - \frac{\alpha}{t} + \frac{\alpha\xi}{t^{2}}\right|, \frac{2\xi}{\lambda(t)},\: \frac{2}{\xi}\right\} \quad \forall t\geq t_{0},
\]
then we have $P\in L_{\text{loc}}^{1}([t_{0}, +\infty), \R)$ and 
\[
\|F(t, z)\| \leq P(t)(1 + \|z\|).
\]

We have checked that the conditions of \cite[Proposition 6.2.1]{Haraux} hold. Therefore, there exists a unique locally absolutely continuous solution $t\mapsto x(t)$ of $(\ref{mainsystem})$ that satisfies $x(t_{0}) = x_{0}$ and $\dot{x}(t_{0}) = u_0$.

\textbf{The case} $\xi = 0$. Now, $(\ref{mainsystem})$ is easily seen to be equivalent to 
\[
\begin{cases}
\dot{z}(t) = F(t, z(t)), \\
z(t_{0}) = (x_{0}, u_{0}),
\end{cases},
\]
where $z(t) = (x(t), y(t))$ and $F$ is given, for every $t\geq t_{0}$, by
\[
F(t, (x, y)) = \left[y, \: -\frac{\alpha}{t}y - T_{\lambda(t), \gamma(t)}(x)\right].
\]
Showing that $F$ fulfills the required properties is starightforward.  
\end{proof}

\section{The convergence properties of the trajectories}\label{sec3}

In this section, we will study the asymptotic behaviour of the trajectories of the system 
\begin{equation*}
    \text{(Split-DIN-AVD)} \quad\ddot{x}(t) + \frac{\alpha}{t}\dot{x}(t) + \xi \frac{d}{dt}\left(T_{\lambda(t), \gamma(t)}(x(t))\right) + T_{\lambda(t), \gamma(t)}(x(t)) = 0,
\end{equation*}
where 
\begin{equation*}
    T_{\lambda, \gamma}(x) = \frac{1}{\lambda}\big[\id - J_{\gamma A}\circ(\id - \gamma B)\big].
\end{equation*}
We will show weak convergence of the trajectories generated by $(\ref{mainsystem})$ to elements of $\zer(A + B)$, as well as the fast convergence of the velocities and accelerations to zero. Additionally, we will provide convergence rates for $T_{\lambda(t), \gamma(t)}(x(t))$ and $\frac{d}{dt}T_{\lambda(t), \gamma(t)}(x(t))$ as $t\to +\infty$. To avoid repetition of the statement ``for almost every $t$'', in the following theorem we will assume we are working with a classic global solution of our system. 

\begin{theorem}\label{maintheorem}
Let $A : \mathcal{H}\to 2^{\mathcal{H}}$ be a maximally monotone operator and $B : \mathcal{H}\to \mathcal{H}$ a $\beta$-cocoercive operator for some $\beta > 0$ such that $\zer(A + B)\neq \emptyset$. Assume that $\alpha > 1$, $\xi \geq 0$, $\lambda(t) = \lambda t^{2}$ for $\lambda > \frac{2}{(\alpha - 1)^{2}}$ and all $t\geq t_{0}$, and that $\gamma: [t_{0}, +\infty)\to (0, 2\beta)$ is a differentiable function that satisfies $\frac{\dot{\gamma}(t)}{\gamma(t)} = \mathcal{O}\left(\frac{1}{t}\right)$ as $t\to +\infty$. Then, for a solution $x : [t_{0}, +\infty) \to \mathcal{H}$ to \textnormal{(Split-DIN-AVD)}, the following statements hold:
\begin{enumerate}[(i)]
    \item $x$ is bounded.
    \item We have the estimates
    \begin{gather*}
    \int_{t_{0}}^{+\infty}t\|\dot{x}(t)\|^{2}dt < +\infty, \quad \int_{t_{0}}^{+\infty}t^{3}\|\ddot{x}(t)\|^{2}dt < +\infty, \\ 
    \int_{t_{0}}^{+\infty}\frac{\gamma^{2}(t)}{t}\left\|A_{\gamma(t)}\Big[x(t) - \gamma(t)Bx(t)\Big] + Bx(t)\right\|^{2}dt < +\infty.
    \end{gather*}
    \item We have the convergence rates
    \begin{align*}
    & \|\dot{x}(t)\| = o\left(\frac{1}{t}\right),  \|\ddot{x}(t)\| = \mathcal{O}\left(\frac{1}{t^{2}}\right),\\
    & \left\|A_{\gamma(t)}\Big[x(t) - \gamma(t)Bx(t)\Big] + Bx(t)\right\| = o\left(\frac{1}{\gamma(t)}\right), \\ 
& \left\|\frac{d}{dt}  \left(A_{\gamma(t)}\Big[x(t) - \gamma(t)Bx(t)\Big] + Bx(t)\right)\right\| = \mathcal{O}\left(\frac{1}{t\gamma(t)}\right) + o\left(\frac{t^{2}\left|\frac{d}{dt}\frac{\gamma(t)}{\lambda(t)}\right|}{\gamma^{2}(t)}\right) 
    \end{align*}
    as $t\to +\infty$.
    \item If $0 < \inf_{t \geq t_{0}}\gamma(t) \leq \sup_{t\geq t_{0}}\gamma(t) < 2\beta$, then $x(t)$ converges weakly to an element of $\zer (A + B)$ as $t \rightarrow +\infty$.
\end{enumerate}
\end{theorem}
\begin{proof}
\textbf{Integral estimates and rates.} To develop the analysis, we will fix $\overline{x}\in \zer(A + B)$ and make of use of the Lyapunov function $\mathcal{E} : [t_{0}, +\infty) \to \R\cup\{+\infty\}$ given by
\begin{equation}
    \mathcal{E}(t) := \frac{1}{2}\left\|\frac{\alpha - 1}{2}(x(t) - \overline{x}) + t(\dot{x}(t) + \xi\:T_{\lambda(t), \gamma(t)}(x(t)))\right\|^{2} + \frac{(\alpha - 1)^{2}}{8}\|x(t) - \overline{x}\|^{2}.
\end{equation}
Differentiation of $\mathcal{E}$ with respect to time yields, for every $t\geq t_{0}$,
\begin{equation*}
    \begin{split}
        \dot{\mathcal{E}}(t) = & \ \left\langle \frac{\alpha - 1}{2}(x(t) - \overline{x}) + t\left(\dot{x}(t) + \xi\: T_{\lambda(t), \lambda(t)}(x(t))\right) \right. , \\
        & \ \ \ \left. \frac{\alpha + 1}{2}\dot{x}(t) + \xi\: T_{\lambda(t), \gamma(t)}(x(t)) + t\left(\ddot{x}(t) + \xi\frac{d}{dt}\left(T_{\lambda(t), \gamma(t)}(x(t))\right)\right)\right\rangle \\
        & \ + \frac{(\alpha - 1)^{2}}{4}\langle x(t) - \overline{x}, \dot{x}(t)\rangle.
    \end{split}
\end{equation*}
After reduction and employing $(\ref{mainsystem})$, we get, for every $t\geq t_{0}$,
\begin{equation*}
    \begin{split}
        \dot{\mathcal{E}}(t) = & \ \frac{(\alpha - 1)(\xi - t)}{2}\left\langle x(t) - \overline{x}, T_{\lambda(t), \gamma(t)}(x(t))\right\rangle + \frac{(1 - \alpha)t}{2}\|\dot{x}(t)\|^{2} \\
        & \ + \left(-t^{2} + \frac{\xi(3 - \alpha)t}{2}\right)\left\langle T_{\lambda(t), \gamma(t)}(x(t)), \dot{x}(t)\right\rangle + \xi(\xi - t)t\left\|T_{\lambda(t), \gamma(t)}(x(t))\right\|^{2}. 
    \end{split}
\end{equation*}
Now, by Lemma \ref{lemma1}(i), we know that $T_{\lambda(t), \gamma(t)}$ is $\frac{\lambda(t)}{2}$-cocoercive for every $t\geq t_{0}$. Using this on the first summand of the right hand side of the previous inequality yields, for $t\geq t_{1} = \max\{\xi, t_{0}\}$, 
\begin{equation}\label{aux5}
    \begin{split}
        \dot{\mathcal{E}}(t) \leq & \ \frac{(1 - \alpha)t}{2}\|\dot{x}(t)\|^{2} + \left(-t^{2} + \frac{\xi(3 - \alpha)t}{2}\right)\left\langle T_{\lambda(t), \gamma(t)}(x(t)), \dot{x}(t)\right\rangle \\
        & \ + \left(\frac{(\alpha - 1)(\xi - t)\lambda(t)}{4} + \xi(\xi - t)t\right)\left\|T_{\lambda(t), \gamma(t)}(x(t))\right\|^{2}.
    \end{split}
\end{equation}
Now, since $\lambda > \frac{2}{(\alpha - 1)^{2}}$, we can choose $\epsilon > 0$ such that 
\begin{equation}\label{aux6}
0 < \epsilon < \alpha - 1 - \sqrt{\frac{2}{\lambda}} < \alpha - 1.
\end{equation}
From $(\ref{aux5})$ we get, for every $t\geq t_{1}$,
\begin{equation}\label{aux7}
    \begin{split}
        & \ \dot{\mathcal{E}}(t) \, +\,  \frac{\epsilon}{2} t\|\dot{x}(t)\|^{2} + \frac{\epsilon}{4} t\lambda(t)\left\|T_{\lambda(t), \gamma(t)}(x(t))\right\|^{2} \\
        \leq & \ \left(\frac{1 - \alpha}{2} + \frac{\epsilon}{2}\right)t\|\dot{x}(t)\|^{2} + \left(-t^{2} + \frac{\xi(3 - \alpha)t}{2}\right)\left\langle T_{\lambda(t), \gamma(t)}(x(t)), \dot{x}(t)\right\rangle \\
        & \  +\left( \left(\frac{(\alpha - 1)(\xi - t)}{2} + \frac{\epsilon}{2}t\right)\frac{\lambda(t)}{2} + \xi(\xi - t)t\right)\left\|T_{\lambda(t), \gamma(t)}(x(t))\right\|^{2}.
    \end{split}
\end{equation}
By $(\ref{aux6})$ and the definition of $\lambda(t)$, we know that $\frac{1 - \alpha}{2} + \frac{\epsilon}{2} < 0$, and 
\[
\left( \left(\frac{(\alpha - 1)(\xi - t)}{2} + \frac{\epsilon}{2}t\right)\frac{\lambda(t)}{2} + \xi(\xi - t)t\right) = \underbrace{\left(\frac{1 - \alpha}{2} + \frac{\epsilon}{2}\right)}_{< 0}\frac{\lambda}{2} t^{3} + \mathcal{O}(t^{2}),
\]
so we can find $t_{2}\geq t_{1}$ such that for every $t\geq t_{2}$ the previous expression becomes nonpositive. According to Lemma \ref{A3}, the right hand side of $(\ref{aux7})$ is nonpositive whenever
\[
R(t) := \left(-t^{2} + \frac{\xi(3 - \alpha)t}{2}\right)^{2} - 4\left(\frac{1 - \alpha}{2} + \frac{\epsilon}{2}\right)t\left(\left(\frac{(\alpha - 1)(\xi - t)}{2} + \frac{\epsilon}{2}t\right)\frac{\lambda(t)}{2} + \xi(\xi - t)t\right) \leq 0. 
\]
This quantity can be rewritten as
\[
R(t) = \left(1 + 4\left(\frac{1 - \alpha}{2} + \frac{\epsilon}{2}\right)\left(\frac{\alpha - 1}{2} - \frac{\epsilon}{2}\right)\frac{\lambda}{2}\right)t^{4} + \mathcal{O}(t^{3}) \quad \text{as} \quad t\to +\infty.
\]
Since $\epsilon < \alpha - 1 - \sqrt{\frac{2}{\lambda}}$, we have $\frac{\lambda}{2} > \frac{1}{(\alpha - 1 - \epsilon)^{2}}$. Hence, 
\[
1 + 4\left(\frac{1 - \alpha}{2} + \frac{\epsilon}{2}\right)\left(\frac{\alpha - 1}{2} - \frac{\epsilon}{2}\right)\frac{\lambda}{2} = 1 - (\alpha - 1 - \epsilon)^{2}\frac{\lambda}{2} < 0.
\]
This means we can find $t_{3}\geq t_{2}$ such that for every $t\geq t_{3}$ we have $R(t) \leq 0$, that is, for every $t\geq t_{3}$ we have
\begin{equation}\label{aux8}
   \dot{\mathcal{E}}(t) + \frac{\epsilon}{2} t\|\dot{x}(t)\|^{2} + \frac{\epsilon}{4} t\lambda(t)\left\|T_{\lambda(t), \gamma(t)}(x(t))\right\|^{2} \leq 0.
\end{equation}
Now, integrating $(\ref{aux8})$ from $t_{3}$ to $t$ we obtain
\begin{equation}\label{aux9}
    \mathcal{E}(t) + \frac{\epsilon}{2}\int_{t_{3}}^{t}s\|\dot{x}(s)\|^{2}ds + \frac{\epsilon}{4}\lambda\int_{t_{3}}^{t}s^{3}\left\|T_{\lambda(s), \gamma(s)}(x(s))\right\|^{2}ds \leq \mathcal{E}(t_{3}).
\end{equation}
From $(\ref{aux8})$ and the form of $\mathcal{E}$ we immediately obtain 
\begin{align}
    &t\mapsto \|x(t) - \overline{x}\| \:\:\text{is bounded}, \label{aux10}\\
    &\int_{t_{0}}^{+\infty}t\|\dot{x}(t)\|^{2}dt < +\infty, \label{aux11}\\
    &\int_{t_{0}}^{+\infty}t^{3}\left\|T_{\lambda(t), \gamma(t)}(x(t))\right\|^{2}dt < +\infty, \label{aux12} \\
    &\sup_{t\geq t_{0}}\left\|\left(\frac{\alpha - 1}{2}\right)(x(t) - \overline{x}) + t\left(\dot{x}(t) + \xi \:T_{\lambda(t), \gamma(t)}(x(t))\right)\right\| < +\infty. \label{aux13}
\end{align}
From Lemma \ref{lemma1}(i), we know that for every $t\geq t_{0}$ the operator $T_{\lambda(t), \gamma(t)}$ is $\frac{2}{\lambda(t)}$-Lipschitz continuous, which gives, for every $t\geq t_{0}$, 
\[
\left\|T_{\lambda(t), \gamma(t)}(x(t))\right\| = \left\|T_{\lambda(t), \gamma(t)}(x(t)) - T_{\lambda(t), \gamma(t)}(\overline{x})\right\| \leq \frac{2}{\lambda(t)}\|x(t) - \overline{x}\|.
\]
Thus, from $(\ref{aux10})$ and recalling that $\lambda(t) = \lambda t^{2}$ we arrive at
\begin{equation}\label{aux14}
    \left\|T_{\lambda(t), \gamma(t)}(x(t))\right\| = \mathcal{O}\left(\frac{1}{t^{2}}\right) \quad \text{as}\quad t\to +\infty.
\end{equation}
By combining $(\ref{aux10})$, $(\ref{aux13})$ and $(\ref{aux14})$ we obtain $\sup_{t\geq t_{0}}t\|\dot{x}(t)\| < +\infty$ and therefore
\begin{equation}\label{aux15}
    \|\dot{x}(t)\| = \mathcal{O}\left(\frac{1}{t}\right) \quad \text{as} \quad t\to +\infty.
\end{equation}
From Lemma \ref{lemma1}, $(\ref{aux10})$, $(\ref{aux15})$ and the fact that $B$ is $\frac{1}{\beta}$-Lipschitz continuous we deduce that, as $t\to +\infty$,
\begin{equation}\label{aux16}
\left\|\frac{d}{dt}\lambda(t)T_{\lambda(t), \gamma(t)}(x(t))\right\| \leq 4\|\dot{x}(t)\| + 4\beta\frac{|\dot{\gamma}(t)|}{\gamma(t)}\|B(x(t))\| + 2\frac{|\dot{\gamma}(t)|}{\gamma(t)}\|x(t) - \overline{x}\| = \mathcal{O}\left(\frac{1}{t}\right).
\end{equation}
On the other hand, for every $t\geq t_{0}$ we have 
\begin{equation}\label{aux17}
\left\|\frac{d}{dt}\lambda(t)T_{\lambda(t), \gamma(t)}(x(t))\right\| = \left\|\dot{\lambda}(t)T_{\lambda(t), \gamma(t)}(x(t)) + \lambda(t)\frac{d}{dt}T_{\lambda(t), \gamma(t)}(x(t))\right\|,
\end{equation}
so by combining $(\ref{aux14})$, $(\ref{aux16})$, $(\ref{aux17})$ and the fact that $\dot{\lambda}(t) = 2\lambda t$ we arrive at
\[
\left\|\lambda(t)\frac{d}{dt}T_{\lambda(t), \gamma(t)}(x(t))\right\| \leq \underbrace{\left\|\frac{d}{dt}\lambda(t)T_{\lambda(t), \gamma(t)}(x(t))\right\|}_{\mathcal{O}\left(\frac{1}{t}\right)} + \dot{\lambda}(t)\underbrace{\left\|T_{\lambda(t), \gamma(t)}(x(t))\right\|}_{\mathcal{O}\left(\frac{1}{t^{2}}\right)} = \mathcal{O}\left(\frac{1}{t}\right) \quad\text{as}\quad t\to +\infty,
\]
which yields 
\begin{equation}\label{aux18}
\left\|\frac{d}{dt}T_{\lambda(t), \gamma(t)}(x(t))\right\| = \frac{1}{\lambda(t)}\mathcal{O}\left(\frac{1}{t}\right) = \mathcal{O}\left(\frac{1}{t^{3}}\right) \quad \text{as} \quad t\to +\infty.
\end{equation}
Let us now improve $(\ref{aux14})$ and show that 
\begin{equation}\label{fast decay of T}
    \left\|T_{\lambda(t), \gamma(t)}(x(t))\right\| = o\left(\frac{1}{t^{2}}\right) \quad \text{as} \quad t\to +\infty.
\end{equation}
According to $(\ref{aux14})$ and $(\ref{aux16})$ there exists a constant $K > 0$ such that for every $t\geq t_{0}$ it holds
\begin{align*}
    \left|\frac{d}{dt}\left\|\lambda(t)T_{\lambda(t), \gamma(t)}(x(t))\right\|^{4}\right| &= \left|4\left\|\lambda(t)T_{\lambda(t), \gamma(t)}(x(t))\right\|^{2} \left\langle \lambda(t)T_{\lambda(t), \gamma(t)}(x(t)), \frac{d}{dt}\lambda(t)T_{\lambda(t), \gamma(t)}(x(t))\right\rangle\right| \\
    &\leq 4 \left\|\lambda(t)T_{\lambda(t), \gamma(t)}(x(t))\right\|^{2} \left\|\lambda(t)T_{\lambda(t), \gamma(t)}(x(t))\right\|\left\|\frac{d}{dt}\lambda(t)T_{\lambda(t), \gamma(t)}(x(t))\right\| \\
    &\leq \frac{4K}{t}\left\|\lambda(t)T_{\lambda(t), \gamma(t)}(x(t))\right\|^{2}.
\end{align*}
By $(\ref{aux12})$, the right hand side belongs to $L^{1}([t_{0}, +\infty), \R)$, so we get
\begin{equation*}
    \frac{d}{dt}\left\|\lambda(t)T_{\lambda(t), \gamma(t)}(x(t))\right\|^{4} \in L^{1}([t_{0}, +\infty), \R), 
\end{equation*}
hence the limit 
\[
\lim_{t\to +\infty}\left\|\lambda(t)T_{\lambda(t), \gamma(t)}(x(t))\right\|^{4}
\]
exists. Obviously, this implies the existence of $L:= \lim_{t\to +\infty}\left\|\lambda(t)T_{\lambda(t), \gamma(t))}(x(t))\right\|^{2}$. By using $(\ref{aux12})$ again we come to 
\[
\int_{t_{0}}^{+\infty}\frac{1}{t}\left\|\lambda(t)T_{\lambda(t), \gamma(t)}(x(t))\right\|^{2}dt = \lambda^{2}\int_{t_{0}}^{+\infty}t^{3}\left\|T_{\lambda(t), \gamma(t)}(x(t))\right\|^{2}dt < +\infty,
\]
and so we must have $L = 0$, which gives
\begin{equation}\label{aux20}
    \left\|T_{\lambda(t), \gamma(t)}(x(t))\right\| = o\left(\frac{1}{t^{2}}\right) \quad \text{as} \quad t\to +\infty.
\end{equation}
By combining $(\ref{mainsystem})$, $(\ref{aux14})$, $(\ref{aux15})$ and $(\ref{aux18})$ we obtain, as $t\to +\infty$,
\begin{align*}
\|\ddot{x}(t)\| &= \left\|-\frac{\alpha}{t}\dot{x}(t) - \xi\frac{d}{dt}T_{\lambda(t), \gamma(t)}(x(t)) - T_{\lambda(t), \gamma(t)}(x(t))\right\| \\
&\leq \frac{\alpha}{t}\underbrace{\|\dot{x}(t)\|}_{\mathcal{O}\left(\frac{1}{t}\right)} + \xi\underbrace{\left\|\frac{d}{dt}T_{\lambda(t), \gamma(t)}(x(t))\right\|}_{\mathcal{O}\left(\frac{1}{t^{3}}\right)} + \underbrace{\left\|T_{\lambda(t), \gamma(t)}(x(t))\right\|}_{\mathcal{O}\left(\frac{1}{t^{2}}\right)} = \mathcal{O}\left(\frac{1}{t^{2}}\right).
\end{align*}
Moreover, by using the well-known inequality $\|a + b + c\|^{2} \leq 3\|a\|^{2} + 3\|b\|^{2} + 3\|c\|^{2}$ for every $a, b, c\in\mathcal{H}$, for every $t\geq t_{0}$ it holds
\begin{align*}
t^{3}\|\ddot{x}(t)\|^{2} &\leq t^{3}\left\|- \frac{\alpha}{t}\dot{x}(t) - \xi\frac{d}{dt}T_{\lambda(t), \gamma(t)}(x(t)) - T_{\lambda(t), \gamma(t)}(x(t))\right\|^{2} \\
&\leq 3\alpha t\|\dot{x}(t)\|^{2} + 3\xi^{2}t^{3}\left\|\frac{d}{dt}T_{\lambda(t), \gamma(t)}(x(t))\right\|^{2} + 3t^{3}\left\|T_{\lambda(t), \gamma(t)}(x(t))\right\|^{2}.
\end{align*}
From $(\ref{aux11})$, $(\ref{aux18})$ and $(\ref{aux12})$ it follows
\begin{equation}\label{aux21}
    \int_{t_{0}}^{+\infty}t^{3}\|\ddot{x}(t)\|^{2}dt < +\infty.
\end{equation}
To see that $\|\dot{x}(t)\| = o\left(\frac{1}{t}\right)$ as $t\to +\infty$, we write, for every $t\geq t_{0}$,
\begin{align*}
    \frac{d}{dt} \left(t^{2}\|\dot{x}(t)\|^{2} \right) = 2t\|\dot{x}(t)\|^{2} + 2t^{2}\langle \dot{x}(t), \ddot{x}(t)\rangle \leq 3t\|\dot{x}(t)\|^{2} + t^{3}\|\ddot{x}(t)\|^{2}.
\end{align*}
From $(\ref{aux11})$ and $(\ref{aux21})$ we deduce that the left hand side belongs to $L^{1}([t_{0}, +\infty), \R)$, from which we infer that the limit $\lim_{t\to +\infty}t^{2}\|\dot{x}(t)\|^{2}$ exists. Using $(\ref{aux11})$ again, we get 
\[
\int_{t_{0}}^{+\infty}\frac{1}{t}\left(t^{2}\|\dot{x}(t)\|^{2}\right)dt = \int_{t_{0}}^{+\infty}t\|\dot{x}(t)\|^{2}dt < +\infty,
\]
from which we finally deduce $\lim_{t\to +\infty}t^{2}\|\dot{x}(t)\|^{2} = 0$, therefore 
\begin{equation}\label{fast decay of x'}
\|\dot{x}(t)\| = o\left(\frac{1}{t}\right) \quad \text{as} \quad t\to +\infty.
\end{equation}

Notice that we can write for every $t\geq t_{0}$
\[
T_{\lambda(t), \gamma(t)} = \frac{1}{\lambda(t)}\Big[\id - J_{\gamma(t)A}(\id - \gamma(t)B)\Big] = \frac{\gamma(t)}{\lambda(t)}\left(A_{\gamma(t)}\Big[x(t) - \gamma(t)Bx(t)\Big] + Bx(t)\right).
\]
Hence, multiplying both sides of $(\ref{aux20})$ by $\frac{\lambda(t)}{\gamma(t)}$ and remembering the definition of $\lambda(t)$ we obtain
\begin{equation}\label{aux65}
\left\|A_{\gamma(t)}\Big[x(t) - \gamma(t)Bx(t)\Big] + Bx(t)\right\| = o\left(\frac{1}{\gamma(t)}\right) \quad \text{as} \quad t\to +\infty.
\end{equation}
For every $t\geq t_{0}$, we have
\begin{align*}
    \frac{d}{dt}T_{\lambda(t), \gamma(t)}(x(t)) = &\: \frac{d}{dt}\left(\frac{\gamma(t)}{\lambda(t)}\right)\left(A_{\gamma(t)}\Big[x(t) - \gamma(t)Bx(t)\Big] + Bx(t)\right) \\
    &+ \frac{\gamma(t)}{\lambda(t)}\frac{d}{dt}\left(A_{\gamma(t)}\Big[x(t) - \gamma(t)Bx(t)\Big] + Bx(t)\right).
\end{align*}
Therefore, by using $(\ref{aux18})$ and $(\ref{aux65})$, and recalling that $\lambda(t) = \lambda t^{2}$, we obtain
\[
\left\|\frac{d}{dt}\left(A_{\gamma(t)}\Big[x(t) - \gamma(t)Bx(t)\Big] + Bx(t)\right)\right\| = \mathcal{O}\left(\frac{1}{t\gamma(t)}\right) + o\left(\frac{t^{2}\left|\frac{d}{dt}\frac{\gamma(t)}{\lambda(t)}\right|}{\gamma^{2}(t)}\right) \quad \text{as} \quad t\to +\infty.
\]

The fact that $\|\ddot{x}(t)\| = \mathcal{O}\left(\frac{1}{t^{2}}\right)$ as $t\to +\infty$ comes from $(\ref{mainsystem})$, $(\ref{fast decay of x'})$, $(\ref{aux18})$ and $(\ref{fast decay of T})$.  

\noindent\textbf{Weak convergence of the trajectories.} Let $\overline{x}\in\zer(A + B)$. We will work with the energy function $h :[t_{0}, +\infty)\to \R$ given by
\[
h(t) := \frac{1}{2}\|x(t) - \overline{x}\|^{2}.
\]
For every $t\geq t_{0}$, we have
\begin{equation}\label{aux22}
    \dot{h}(t) = \langle x(t) - \overline{x}, \dot{x}(t)\rangle, \quad \ddot{h}(t) = \langle x(t) - \overline{x}, \ddot{x}(t)\rangle + \|\dot{x}(t)\|^{2}.
\end{equation}
Combining $(\ref{mainsystem})$ and $(\ref{aux22})$ gives us, for every $t\geq t_{0}$,
\begin{equation*}
    \ddot{h}(t) + \frac{\alpha}{t}\dot{h}(t) + \left\langle T_{\lambda(t), \gamma(t)}(x(t)), x(t) - \overline{x}\right\rangle = \|\dot{x}(t)\|^{2} + \left\langle -\xi\frac{d}{dt}T_{\lambda(t), \gamma(t)}(x(t)), x(t) - \overline{x}\right\rangle.
\end{equation*}
By using the $\frac{\lambda(t)}{2}$-cocoercitivity of $T_{\lambda(t), \gamma(t)}$ on the left hand side, Cauchy-Schwarz on the right hand side and multiplying both sides by $t$, the previous inequality entails, for every $t\geq t_{0}$,
\begin{equation*}
    t\ddot{h}(t) + \alpha\dot{h}(t) +  t \frac{\lambda(t)}{2}\left\|T_{\lambda(t), \gamma(t)}(x(t))\right\| \leq t\|\dot{x}(t)\|^{2} + \xi t \left\|\frac{d}{dt}T_{\lambda(t), \gamma(t)}(x(t))\right\|\|x(t) - \overline{x}\| \quad \forall t\geq t_{0}.
\end{equation*}
Now, putting $(\ref{aux10})$, $(\ref{aux11})$ and $(\ref{aux18})$ together results in 
\[
k(t) := t\|\dot{x}(t)\|^{2} + \xi t \left\|\frac{d}{dt}T_{\lambda(t), \gamma(t)}(x(t))\right\|\|x(t) - \overline{x}\| \in L^{1}([t_{0}, +\infty), \R).
\]
Now apply Lemma \ref{A2} with $\theta(t):= t \frac{\lambda(t)}{2}\left\|T_{\lambda(t), \gamma(t)}(x(t))\right\|$ for every $t\geq t_{0}$ to deduce that the limit
\[
\lim_{t\to +\infty}h(t)
\]
exists, which fulfills the first condition of Opial's Lemma \ref{opial lemma}.

Let us now move on to the second condition. Suppose $\widehat{x}$ is a weak sequential cluster point of $t\mapsto x(t)$, that is, there exists a sequence $(t_{n})_{n\in\N}\subseteq [t_{0}, +\infty)$ such that $t_{n}\to +\infty$ and $x_{n} := x(t_{n})$ converges weakly to $\widehat{x}$ as $n\to +\infty$. Define
\[
U_{\gamma} := \id - J_{\gamma A}\circ(\id - \gamma B).
\]
According to $(\ref{aux20})$, we have $U_{\gamma(t)}(x(t)) = \lambda(t)T_{\lambda(t), \gamma(t)}(x(t))\to 0$ as $t\to +\infty$. Now, since $\gamma(t)\in[\delta, 2\beta - \delta]$ for all $t\geq t_{0}$ for some $\delta > 0$, we can extract a subsequence $(\gamma(t_{n_{k}}))_{k\in\N}$ such that $\gamma(t_{n_{k}})\to \overline{\gamma}\in (0, 2\beta)$ as $k\to +\infty$. We may assume without loss of generality then that $\gamma_{n} := \gamma(t_{n})\to \overline{\gamma}$ as $n\to +\infty$. We now have for every $n \in \N$
\begin{align*}
    \|U_{\gamma_{n}}(x_{n}) - U_{\overline{\gamma}}(x_{n})\| = & \ \|J_{\gamma_{n} A}(x_{n} - \gamma_{n}B(x_{n})) - J_{\overline{\gamma}A}(x_{n} - \overline{\gamma}B(x_{n}))\| \\
    = & \ \|J_{\gamma_{n}A}(x_{n} - \gamma_{n}B(x_{n})) - J_{\gamma_{n}A}(x_{n} - \overline{\gamma}B(x_{n}))\| \\
    & + \|J_{\gamma_{n}A}(x_{n} - \overline{\gamma}B(x_{n})) - J_{\overline{\gamma}A}(x_{n} - \overline{\gamma}B(x_{n}))\| \\
    \leq & \ |\overline{\gamma} - \gamma_{n}|\|B(x_{n})\| + |\overline{\gamma} - \gamma_{n}|\|A_{\overline{\gamma}}(x_{n} - \overline{\gamma}B(x_{n}))\|.
\end{align*}
Now, since every weakly convergent sequence is bounded and the operators $B$ and $A_{\overline{\gamma}}$ are Lipschitz-continuous we deduce that the right-hand side of the previous inequality approaches zero as $n\to +\infty$, therefore getting
\[
U_{\overline{\gamma}}(x_{n}) = U_{\gamma_{n}}(x_{n}) + \big(U_{\overline{\gamma}}(x_{n}) - U_{\gamma_{n}}(x_{n})\big) \to 0
\]
as $n\to +\infty$. Now, from the proof of part (i) of Lemma \ref{lemma1}, we know that $U_{\overline{\gamma}}$ is $\frac{4\beta - \overline{\gamma}}{4\beta}$-cocoercive, thus monotone and Lipschitz continuous and therefore maximally monotone. Summarizing, we have
\begin{enumerate}
    \item $U_{\overline{\gamma}}$ is maximally monotone and thus its graph is closed in the weak$\times$strong topology of $\mathcal{H} \times\mathcal{H}$ (see \cite[Proposition 20.38(ii)]{BC}), 
    \item $x_{n}$ converges weakly to $\widehat{x}$ and $U_{\overline{\gamma}}(x_{n})\to 0$ as $n\to +\infty$,
\end{enumerate}
which allows us to conclude that $U_{\overline{\gamma}}(\widehat{x}) = 0$, and gives finally $\widehat{x}\in\zer(A + B)$. Now we just invoke Opial's Lemma to achieve that $x(t)$ converges weakly to $\overline{x}$ as $t\to +\infty$ for some $\overline{x}\in\zer(A + B)$.
\end{proof}

In the following subsections, we explore the particular cases $B = 0$ and $A = 0$, and we will show improvements with respect to previous results from the literature addressing continuous time approaches to monotone inclusions. 

\subsection{The case $B = 0$}
If we let $B = 0$ in the (Split-DIN-AVD) system $(\ref{mainsystem})$, then, attached to the monotone inclusion problem 
$$\mbox{find} \ x \in \mathcal{H} \ \mbox{such that} \ 0 \in A(x),$$ 
we obtain the dynamics 
\begin{equation}\label{system3}
    \ddot{x}(t) + \frac{\alpha}{t}\dot{x}(t) + \xi \frac{d}{dt}\left(A_{\lambda(t), \gamma(t)}(x(t)\right) + A_{\lambda(t), \gamma(t)}(x(t)) = 0,
\end{equation}
where 
\[
A_{\lambda, \gamma}(x) = \frac{1}{\lambda}(\id - J_{\gamma A}).
\]
We can state the following theorem.
\begin{theorem}\label{TheoremCaseBzero}
Let $A : \mathcal{H}\to 2^{\mathcal{H}}$ be a maximally monotone operator such that $\zer A\neq \emptyset$. Assume that $\alpha > 1$, $\xi\geq 0$, $\lambda(t) = \lambda t^{2}$ for $\lambda > \frac{1}{(\alpha - 1)^{2}}$ and all $t\geq t_{0}$, and that $\gamma: [t_{0}, +\infty)\to (0, +\infty)$ is a differentiable function that satisfies $\frac{|\dot{\gamma}(t)|}{\gamma(t)} = \mathcal{O}\left(\frac{1}{t}\right)$ as $t\to +\infty$. Then, for a solution $x : [t_{0}, +\infty) \to \mathcal{H}$ to $(\ref{system3})$, the following statements hold: 
\begin{enumerate}[(i)]
    \item $x$ is bounded.
    \item We have the estimates
    \[
    \int_{t_{0}}^{+\infty}t\|\dot{x}(t)\|^{2}dt < +\infty, \quad \int_{t_{0}}^{+\infty}t^{3}\|\ddot{x}(t)\|^{2}dt < +\infty,\quad \int_{t_{0}}^{+\infty}\frac{\gamma^{2}(t)}{t}\left\|A_{\gamma(t)}(x(t))\right\|^{2}dt < +\infty.
    \]
    \item We have the convergence rates
    \begin{align*}
   & \|\dot{x}(t)\| = o\left(\frac{1}{t}\right),  \ \|\ddot{x}(t)\| = \mathcal{O}\left(\frac{1}{t^{2}}\right), \\
    & \left\|A_{\gamma(t)}(x(t))\right\| = o\left(\frac{1}{\gamma(t)}\right), \ \left\|\frac{d}{dt}A_{\gamma(t)}(x(t))\right\| = \mathcal{O}\left(\frac{1}{t\gamma(t)}\right) + o\left(\frac{t^{2}\left|\frac{d}{dt}\frac{\gamma(t)}{\lambda(t)}\right|}{\gamma^{2}(t)}\right)
    \end{align*}
    as $t\to +\infty$.
    \item If $0 < \inf_{t\geq t_{0}}\gamma(t)$, then $x(t)$ converges weakly to an element of $\zer A$ as $t \rightarrow +\infty$.
\end{enumerate}
\end{theorem}
\begin{proof}
The proof proceeds in the exact same way as the proof of Theorem \ref{maintheorem}. However, a few comments are in order: first of all, now we have $T_{\lambda, \gamma} = \frac{1}{\lambda}(\id - J_{\gamma A}) = A_{\lambda, \gamma}$. Since $J_{\lambda A}$ is firmly nonexpansive, by \cite[Proposition 4.4]{BC} so is $\id - J_{\lambda A}$. In other words, $\id - J_{\gamma A}$ is $1$-cocoercive, therefore $A_{\lambda, \gamma} = \frac{1}{\lambda}(\id - J_{\gamma A})$ is $\lambda$-cocoercive, so now the condition on $\lambda$ becomes $\lambda > \frac{1}{(\alpha - 1)^{2}}$. 

The proof also changes when we verify the second part of the Opial's Lemma, to get weak convergence of the trajectories $t\mapsto x(t)$. This is in order to allow for $\gamma(t)$ not to be necessarily bounded. We do need, however, the assumption $0 < \inf_{t\geq t_{0}}\gamma(t)$. Indeed, from $\|A_{\lambda(t), \gamma(t)}(x(t))\| = o\left(\frac{1}{t^{2}}\right)$ as $t\to +\infty$, we obtain
\[
y(t) := x(t) - J_{\gamma(t)A}x(t) = \lambda(t)A_{\lambda(t), \gamma(t)}(x(t)) \to 0
\]
as $t\to +\infty$. Using the definition of the resolvent, we come to
\[
J_{\gamma(t)A}x(t) = x(t) - y(t) \Leftrightarrow y(t) \in \gamma(t)A(x(t) - y(t)) \Leftrightarrow \frac{1}{\gamma(t)}y(t) \in A(x(t) - y(t)).
\]
for all $t\geq t_{0}$. If $(t_{n})_{n\in\N}\subseteq [t_{0}, +\infty)$ is such that $t_{n}\to +\infty$ and $x(t_{n})$ converges weakly to $\widehat{x}$ as $n\to +\infty$, then the previous inclusion, together with the assumption on $\gamma$ gives
\[
x(t_{n}) - y(t_{n}) \ \mbox{converges weakly to} \ \widehat{x} \quad \text{and}\quad \frac{1}{\gamma(t)}y(t)\to 0 \quad \text{as}\quad n\to +\infty,
\]
and by the closedness of the graph of $A$ in the weak$\times$strong topology of $\mathcal{H} \times\mathcal{H}$, we deduce that $\widehat{x}\in \zer A$.   
\end{proof}
\begin{remark}\label{remark1}
The hypotheses required for $\gamma$ are fulfilled at least by two families of functions. First, take $r\geq 0$ and set $\gamma(t) = e^{t^{-r}}$. Then, we have
\[
\frac{\dot{\gamma}(t)}{\gamma(t)} = \frac{-r t^{-(r + 1)}e^{t^{-r}}}{e^{t^{-r}}} = -\frac{r}{t^{r + 1}} = \mathcal{O}\left(\frac{1}{t}\right) \quad \text{as} \quad t\to +\infty,
\]
and
\[
\gamma(t) = e^{t^{-r}} \geq e^{0} = 1 \quad \forall t\geq 0.
\]

If $\gamma$ is a polynomial of degree $n$ for some $n\in\N$, the conditions are also fulfilled. Assume $\gamma(t) = a_{n}t^{n} + a_{n - 1}t^{n - 1} + \cdots + a_{0}$ for all $t\geq t_{0}$, for some $a_{i}\in \R$ for $i\in\{0, \ldots, n\}$ and $a_{n} > 0$. Then, we have
\begin{align*}
t\cdot\frac{\dot{\gamma}(t)}{\gamma(t)} &= t \cdot \frac{n a_{n}t^{n - 1} + (n - 1)a_{n - 1}t^{n - 1} + \cdots + a_{1}}{a_{n}t^{n} + a_{n - 1}t^{n - 1} + \cdots + a_{0}} \\
& \to \frac{na_{n}}{a_{n}} = n \quad \text{as} \quad t\to +\infty,
\end{align*}
so $\frac{\dot{\gamma}(t)}{\gamma(t)} = \mathcal{O}\left(\frac{1}{t}\right)$ as $t\to +\infty$. Since we also have $\gamma(t) \to +\infty$ as $t\to +\infty$, the condition $\inf_{t\geq t_{0}} \gamma(t) > 0$ is fulfilled for large enough $t_{0}$. 

In particular, we can choose $\gamma(t) = \lambda(t) = \lambda t^{2}$, which fulfills $\gamma(t) \geq \lambda t_{0}^{2} > 0$ for any $t\geq t_{0}$ and any $t_{0}$. Since $A_{\lambda, \lambda} = A_{\lambda}$ for $\lambda > 0$, this choice of $\gamma$ allows us to recover the (DIN-AVD) system studied by Attouch and L\'aszl\'o in \cite{AttouchLaszlo}. Notice the way the convergence rates for $A_{\gamma(t)}(x(t))$ and $\frac{d}{dt}A_{\gamma(t)}(x(t))$ exhibited in part (iii) of Theorem \ref{TheoremCaseBzero} depend on $\gamma(t)$. If we set $\gamma(t) = t^{n}$ for every $t\geq t_{0}$ for any natural number $n > 2$, (Split-DIN-AVD) performs from this point of view better than (DIN-AVD) without increasing the complexity of the governing operator. 
\end{remark}

\subsection{The case $A = 0$}
Let us return to (Split-DIN-AVD) dynamics $(\ref{mainsystem})$. Set $A = 0$, and for every $t\geq t_{0}$ take $\gamma(t) = \gamma \in (0, 2\beta)$ and $\eta(t) = \eta t^{2}$ with $\eta = \lambda/\gamma$. Then, associated to the problem 
\[
\mbox{find} \ x \in \mathcal{H} \ \mbox{such that} \ B(x)=0,
\] 
we obtain the system
\begin{equation}\label{system2}
    \ddot{x}(t) + \frac{\alpha}{t}\dot{x}(t) + \xi \frac{d}{dt}\left(\frac{1}{\eta(t)}Bx(t)\right) + \frac{1}{\eta(t)}Bx(t) = 0.
\end{equation}
The conditions $\lambda > \frac{2}{(\alpha - 1)^{2}}$ and $\gamma\in (0, 2\beta)$ imply
\[
\eta = \frac{\lambda}{\gamma} > \frac{2}{\gamma(\alpha - 1)^{2}} > \frac{2}{2\beta(\alpha - 1)^{2}} = \frac{1}{\beta(\alpha - 1)^{2}}.
\]
With the previous observation, we are able to state the following theorem.
\begin{theorem}\label{TheoremCaseAzero}
Let $B: \mathcal{H}\to \mathcal{H}$ be a $\beta$-cocoercive operator for some $\beta > 0$ such that $\zer B\neq \emptyset$. Assume that $\alpha > 1$, $\xi \geq 0$ and $\eta(t) = \eta t^{2}$ for $\eta > \frac{1}{\beta(\alpha - 1)^{2}}$ and all $t\geq t_{0}$. Take $x : [t_{0}, +\infty)\to \mathcal{H}$ a solution to $(\ref{system2})$. Then, the following hold:
\begin{enumerate}[(i)]
    \item $x$ is bounded, and $x(t)$ converges weakly to an element of $\zer B$ as $t \rightarrow +\infty$.
    \item We have the estimates
    \[
    \int_{t_{0}}^{+\infty}t\|\dot{x}(t)\|^{2}dt < +\infty, \quad \int_{t_{0}}^{+\infty}t^{3}\|\ddot{x}(t)\|^{2}dt < +\infty, \quad \int_{t_{0}}^{+\infty}\frac{1}{t}\left\|Bx(t)\right\|^{2}dt < \infty.
    \]
    \item We have the convergence rates
    \begin{equation*}
    \|\dot{x}(t)\| = o\left(\frac{1}{t}\right), \quad \|\ddot{x}(t)\| = \mathcal{O}\left(\frac{1}{t^{2}}\right)
    \end{equation*}
    as well as the limit
    \[
    \|Bx(t)\| \to 0
    \]
    as $t\to +\infty$.
\end{enumerate}
\end{theorem}
\begin{proof}
Since $\eta > \frac{1}{\beta(\alpha - 1)^{2}}$, we can find $\epsilon\in (0, \beta)$ such that $\eta > \frac{1}{(\beta - \epsilon)(\alpha - 1)^{2}}$, equivalently, $2(\beta - \epsilon)\eta > \frac{2}{(\alpha - 1)^{2}}$. Since $(\ref{system2})$ is equivalent to (Split-DIN-AVD) with $A = 0$ and parameters $\lambda = 2(\beta - \epsilon)\eta > \frac{1}{(\alpha - 1)^{2}}$ and $\gamma(t) \equiv 2(\beta - \epsilon) \in (0, 2\beta)$, the conclusion follows from Theorem \ref{maintheorem}. 
\end{proof}
\begin{remark}
(a) As we mentioned in the introduction, the dynamical system \eqref{system2} provides a way of finding the zeros of a cocoercive  operator directly through forward evaluations, instead of having to resort to its Moreau envelope when following the approach in \cite{AttouchLaszlo}. 

(b) The dynamics $(\ref{system2})$ bear some resemblance to the system $(\ref{systemBotCsetnek})$ (see also \cite{BotCsetnek}) with $\mu(t) = \frac{\alpha}{t}$ and $\nu(t) = \frac{1}{\eta(t)}$, with an additional Hessian-driven damping term. In our case, since $\eta > \frac{1}{\beta(\alpha - 1)^{2}}$, the parameters satisfy
\[
\dot{\mu}(t) = -\frac{\alpha}{t^{2}} \leq 0, \quad \frac{\mu^{2}(t)}{\nu(t)} = \frac{\alpha^{2}\eta t^{2}}{t^{2}} = \alpha^2 \eta> \frac{1}{\beta} \quad \forall t\geq t_{0}.
\]
However, we have
\[
\dot{\nu}(t) = -\frac{2}{\lambda t^{3}} \leq 0 \quad \forall t\geq t_{0},
\]
so one of the hypotheses which is needed in $(\ref{systemBotCsetnek})$ is not fulfilled, which shows that one cannot address the dynamical system $(\ref{system2})$ as a particular case of it; indeed, for $(\ref{systemBotCsetnek})$ a vanishing damping is not allowed. With our system, we obtain convergence rates for $\dot{x}(t)$ and $\ddot{x}(t)$ as $t\to +\infty$, which are not obtained in \cite{BotCsetnek}.

\end{remark}

\section{Structured convex minimization}\label{sec4}

We can specialize the previous results to the case of convex minimization, and show additionally the convergence of functional values along the generated trajectories to the optimal objective value at a rate that will depend on the choice of $\gamma$. Let $f:\mathcal{H}\to \R\cup\{+\infty\}$ be a proper, convex and lower semicontinuous function, and let $g:\mathcal{H}\to \R$ be a convex and Fr\'echet differentiable function with $L_{\nabla g}$-Lipschitz continuous gradient. Assume that $\argmin_{\mathcal{H}}(f + g)\neq\emptyset$, and consider the minimization problem
\begin{equation}\label{ConvexMinimization}
    \min_{x\in\mathcal{H}} f(x) + g(x).
\end{equation}
Fermat's rule tells us that $\overline x$ is a global minimum of $f + g$ if and only if 
\[
0\in\partial (f + g)(\overline x) = \partial f(\overline x) + \nabla g(\overline x).
\]
Therefore, solving $(\ref{ConvexMinimization})$ is equivalent solving the monotone inclusion $0\in (A + B)(x)$ addressed in the first section, with $A = \partial f$ and $B = \nabla g$. Moreover, recall that if $\nabla g$ is $L_{\nabla g}$-Lipschitz then it is $\frac{1}{L_{\nabla g}}$-cocoercive (Baillon-Haddad's Theorem, see \cite[Corollary 18.17]{BC}). Therefore, associated to the problem $(\ref{ConvexMinimization})$ we have the dynamics

\begin{equation} \label{systemFunctions}
    \ddot{x}(t) + \frac{\alpha}{t}\dot{x}(t) + \xi \frac{d}{dt}\left(\frac{\gamma(t)}{\lambda(t)}\left(\nabla f_{\gamma(t)}(u(t)) + \nabla g(x(t))\right)\right) + \frac{\gamma(t)}{\lambda(t)}\left(\nabla f_{\gamma(t)}(u(t)) + \nabla g(x(t)) \right) = 0,
\end{equation}
where we have denoted $u(t) = x(t) - \gamma(t)\nabla g(x(t))$ for all $t\geq t_{0}$ for convenience.

\begin{theorem}\label{minimizar f + g}
Let $f: \mathcal{H}\to \R\cup\{+\infty\}$ be a proper, convex and lower semicontinuous function, and let $g : \mathcal{H}\to \R$ be a convex and Fr\'echet differentiable function with a $L_{\nabla g}$-Lipschitz continuous gradient such that $\argmin_{\mathcal{H}}(f + g)\neq \emptyset$. Assume that $\alpha > 1$, $\xi \geq 0$, $\lambda(t) = \lambda t^{2}$ for $\lambda > \frac{2}{(\alpha - 1)^{2}}$ and all $t\geq t_{0}$, and that $\gamma : [t_{0}, +\infty)\to \left(0, \frac{2}{L_{\nabla g}}\right)$ is a differentiable function that satisfies $\frac{\dot{\gamma}(t)}{\gamma(t)} = \mathcal{O}(1/t)$ as $t\to +\infty$. Then, for a solution $x : [t_{0}, +\infty) \to \mathcal{H}$ to $(\ref{systemFunctions})$, the following statements hold:
\begin{enumerate}[(i)]
    \item $x$ is bounded.
    \item We have the estimates
    \begin{gather*}
    \int_{t_{0}}^{+\infty}t\|\dot{x}(t)\|^{2}dt < +\infty, \quad \int_{t_{0}}^{+\infty}t^{3}\|\ddot{x}(t)\|^{2}dt < +\infty, \\ 
    \int_{t_{0}}^{+\infty} \frac{\gamma^{2}(t)}{t}\left\|\nabla f_{\gamma(t)}\Big[x(t) - \gamma(t)\nabla g(x(t))\Big] + \nabla g(x(t))\right\|^{2}dt < +\infty.
    \end{gather*}
    \item We have the convergence rates 
    \begin{align*}
    & \|\dot{x}(t)\| = o\left(\frac{1}{t}\right), \ \|\ddot{x}(t)\| = \mathcal{O}\left(\frac{1}{t^{2}}\right), \\
    & \left\|\nabla f_{\gamma(t)}\Big[x(t) - \gamma(t)\nabla g(x(t))\Big] + \nabla g(x(t))\right\| = o\left(\frac{1}{\gamma(t)}\right), \\
   & \left\|\frac{d}{dt}\left(\nabla f_{\gamma(t)}\Big[x(t) - \gamma(t)\nabla g(x(t))\Big] + \nabla g(x(t))\right)\right\| = \mathcal{O}\left(\frac{1}{t\gamma(t)}\right) + o\left(\frac{t^{2}\left|\frac{d}{dt}\frac{\gamma(t)}{\lambda(t)}\right|}{\gamma^{2}(t)}\right)
    \end{align*}
    as $t\to +\infty$.
    \item If $0 < \inf_{t \geq t_{0}}\gamma(t) \leq \sup_{t\geq t_{0}}\gamma(t) < \frac{2}{L_{\nabla g}}$, then $x(t)$ converges converges to a minimizer of $f + g$ as $t \rightarrow +\infty$.
    \item Additionally, if $0 < \gamma(t) \leq \frac{1}{L_{\nabla g}}$ for every $t\geq t_{0}$ and we set $u(t) := x(t) - \gamma(t) \nabla g(x(t))$, then 
    \[
    f\left(\prox_{\gamma(t)f}(u(t))\right) + g\left(\prox_{\gamma(t)f}(u(t))\right) - \min\nolimits_{\mathcal{H}}(f + g) = o\left(\frac{1}{\gamma(t)}\right)
    \]
    as $t\to +\infty$. Moreover, $\left\|\prox_{\gamma(t)f}(u(t)) - x(t)\right\|\to 0$ as $t\to +\infty$.
\end{enumerate}
\end{theorem}

\begin{proof}
Parts (i)-(iv) are a direct consequence of Theorem \ref{maintheorem}. For checking (v), first notice that for all $t\geq t_{0}$ we have
\begin{align}
    T_{\lambda(t), \gamma(t)}(x(t)) &= \frac{1}{\lambda(t)}\Big[\id - J_{\gamma(t)\partial f}\circ(\id - \gamma(t)\nabla g)\Big](x(t)) = \frac{1}{\lambda(t)}\Big[x(t) - \prox_{\gamma(t)f}(u(t))\Big]. \label{aux27} 
\end{align}
Now, let $\overline{x}\in \argmin_{\mathcal{H}}(f+g)$. According to \cite[Lemma 2.3]{FISTA}, for every $t\geq t_{0}$, we have the inequality
\begin{align*}
    &\ f\left(\prox_{\gamma(t)f}(u(t))\right) + g\left(\prox_{\gamma(t)f}(u(t))\right) - \min\nolimits_{\mathcal{H}}(f + g) \\
    \leq & \ f\left(\prox_{\gamma(t)f}(u(t))\right) + g\left(\prox_{\gamma(t)f}(u(t))\right) - f(\overline{x}) - g(\overline{x}) \\
    \leq &\ -\frac{1}{2\gamma(t)}\left\|\prox_{\gamma(t)f}(u(t)) - x(t)\right\|^{2} + \frac{1}{\gamma(t)}\left\langle x(t) - x^{*}, x(t) - \prox_{\gamma(t)f}(u(t))\right\rangle.
\end{align*}
After summing the norm squared term and using the Cauchy-Schwarz inequality, for every $t\geq t_{0}$ we obtain
\begin{align*}
    & \ \frac{1}{2\gamma(t)}\left\|\prox_{\gamma(t)f}(u(t)) - x(t)\right\|^{2} \\
    \leq & \ f\left(\prox_{\gamma(t)f}(u(t))\right) + g\left(\prox_{\gamma(t)f}(u(t))\right) + \frac{1}{2\gamma(t)}\left\|\prox_{\gamma(t)f}(u(t)) - x(t)\right\|^{2} - \min\nolimits_{\mathcal{H}}(f + g) \\
    \leq & \ \left\langle\frac{1}{\gamma(t)}\Big(x(t) - \prox_{\gamma(t)f}(u(t))\Big), x(t) - \overline{x}\right\rangle \leq \left\|\frac{1}{\gamma(t)}\Big(x(t) - \prox_{\gamma(t)f}(u(t))\Big)\right\| \|x(t) - \overline{x}\| \\
    = & \ \frac{\lambda(t)}{\gamma(t)}\left\|T_{\lambda(t), \gamma(t)}(x(t))\right\|\|x(t) - \overline{x}\| \\
    = & \ o\left(\frac{1}{\gamma(t)}\right) \quad \text{as} \quad t\to +\infty,
\end{align*}
which follows as a consequence of $x$ being bounded and $\left\|T_{\lambda(t), \gamma(t)}(x(t))\right\| = o\left(\frac{1}{t^{2}}\right)$ as $t\to +\infty$. 
\end{proof}

\begin{remark}
It is also worth mentioning the system we obtain in the case where $g \equiv 0$, since we also get some improved rates for the objective functional values when we compare (Split-DIN-AVD) to (DIN-AVD) \cite{AttouchLaszlo}. In this case, we have the system
\begin{equation}\label{system minimizar f}
    \ddot{x}(t) + \frac{\alpha}{t} + \xi \frac{d}{dt}\left(\frac{\gamma(t)}{\lambda(t)}\nabla f_{\gamma(t)}(x(t))\right) + \frac{\gamma(t)}{\lambda(t)}\nabla f_{\gamma(t)}(x(t)) = 0
\end{equation}
attached to the convex optimization problem 
$$\min_{x\in\mathcal{H}}f(x).$$ 
If we assume $\lambda > \frac{1}{(\alpha - 1)^{2}}$, allow $\gamma : [t_{0}, +\infty) \to (0, +\infty)$ to be unbounded from above and otherwise keep the hypotheses of Theorem \ref{minimizar f + g}, for a solution $x : [t_{0}, +\infty) \to \mathcal{H}$ to $(\ref{system minimizar f})$, the following statements hold:
\begin{enumerate}[(i)]
    \item $x$ is bounded,
    \item We have the estimates
    \[
    \int_{t_{0}}^{+\infty}t\|\dot{x}(t)\|^{2}dt < +\infty, \quad \int_{t_{0}}^{+\infty}t^{3}\|\ddot{x}(t)\|^{2}dt < +\infty,\quad \int_{t_{0}}^{+\infty}\frac{\gamma^{2}(t)}{t}\left\|\nabla f_{\gamma(t)}(x(t))\right\|^{2}dt < +\infty,
    \]
    \item We have the convergence rates
    \begin{align*}
    & \|\dot{x}(t)\| = o\left(\frac{1}{t}\right), \ \|\ddot{x}(t)\| = \mathcal{O}\left(\frac{1}{t^{2}}\right), \\
   & \left\|\nabla f_{\gamma(t)}(x(t))\right\| = o\left(\frac{1}{\gamma(t)}\right), \  \left\|\frac{d}{dt}\nabla f_{\gamma(t)}(x(t))\right\| = \mathcal{O}\left(\frac{1}{t\gamma(t)}\right) + o\left(\frac{t^{2}\left|\frac{d}{dt}\frac{\gamma(t)}{\lambda(t)}\right|}{\gamma^{2}(t)}\right)
    \end{align*}
    as $t\to +\infty$.
    \item If $0 < \inf_{t\geq t_{0}}\gamma(t)$, then $x(t)$ converges weakly to a minimizer of $f$ as $t \rightarrow +\infty$.
    \item We also obtain the rate
    \[
    f_{\gamma(t)}(x(t)) - \min\nolimits_{\mathcal{H}}f = o\left(\frac{1}{\gamma(t)}\right) \quad \text{as} \quad t\to +\infty,
    \]
    which entails
    \[
    f\left(\prox_{\gamma(t)f}(x(t))\right) - \min\nolimits_{\mathcal{H}}f = o\left(\frac{1}{\gamma(t)}\right) \quad \text{and} \quad \left\|\prox_{\gamma(t)f}(x(t)) - x(t)\right\| \to 0
    \]
    as $t\to +\infty$.
\end{enumerate}
Parts (i)-(iv) are a direct consequence of Theorem \ref{TheoremCaseBzero} for the case $A = \partial f$. For showing part (v), first notice that for $\lambda > 0$ and $u\in\mathcal{H}$ we have, according to the definition of $f_{\lambda}$ and $\prox_{\lambda f}$,
\begin{align*}
f_{\lambda}(u) = f\left(\prox_{\lambda f}(u)\right) + \frac{1}{2\lambda}\left\|\prox_{\lambda f}(u) - u\right\|^{2} \leq f(u).
\end{align*}
Let $\overline{x} \in \mathcal{H}$ be a minimizer of $f$. We apply the gradient inequality to $f_{\gamma(t)}$, from which we obtain, for every $t\geq t_{0}$
\begin{align*}
    f_{\gamma(t)}(x(t)) - \min\nolimits_{\mathcal{H}}f &= f_{\gamma(t)}(x(t)) - f(\overline{x}) \leq f_{\gamma(t)}(x(t)) - f_{\lambda(t)}(\overline{x}) \\
& \leq \left\langle \nabla f_{\gamma(t)}(x(t)), x(t) - \overline{x}\right\rangle \leq \left\|\nabla f_{\gamma(t)}(x(t))\right\|\|x(t) - \overline{x}\|,
\end{align*}
where the last inequality follows from the Cauchy-Schwarz inequality. Since $\left\|\nabla f_{\gamma(t)}(x(t))\right\| = o\left(\frac{1}{\gamma(t)}\right)$ as $t\to +\infty$ and $x$ is bounded, the previous inequality entails the first statement of (v). Again recalling the definition of the Moreau envelope of $f$, this finally gives 
\[
f\left(\prox_{\gamma(t)f}(x(t))\right) + \frac{1}{2\gamma(t)}\left\|\prox_{\gamma(t)f}(x(t)) - x(t)\right\|^{2} - \min\nolimits_{\mathcal{H}}f = f_{\gamma(t)}(x(t)) - \min\nolimits_{\mathcal{H}}f = o\left(\frac{1}{\gamma(t)}\right)
\]
as $t\to +\infty$, which implies the last two statements and concludes the proof.

As pointed out in Remark \ref{remark1}, we can choose $\gamma(t) = \lambda t^{2}$ for every $t\geq t_{0}$ and recover the (DIN-AVD) system for nonsmooth convex minimization problems studied in \cite{AttouchLaszlo}. Moreover, we can also set $\gamma(t) = t^{n}$ for a natural number $n > 3$ and all $t\geq t_{0}$. Now, not only are the convergence rates for $\nabla f_{\gamma(t)}(x(t))$ and $\frac{d}{dt}\nabla f_{\gamma(t)}(x(t))$ as $t\to +\infty$ improved with respect to the system in \cite{AttouchLaszlo}, but (Split-DIN-AVD) also provides a better rate for the convergence of $f_{\gamma(t)}(x(t))$ to $\min_{\mathcal{H}}f$ as $t\to +\infty$. 
\end{remark}

\section{Numerical experiments}\label{sec5}

In the following paragraphs we describe some numerical experiments that portray some aspects of the theory.

\subsection{Minimizing a smooth and convex function}

As an example of a continuous time scheme minimizing a convex and Fr\'echet differentiable function $g : \mathcal{H} \rightarrow \R$ with $L_{\nabla g}$-Lipschitz continuous gradient via (Split-DIN-AVD), we consider the system
\begin{equation}\label{system A = 0}
\ddot{x}(t) + \frac{\alpha}{t}\dot{x}(t) + \xi\frac{d}{dt}\left(\frac{1}{\eta(t)}\nabla g(x(t))\right) + \frac{1}{\eta(t)}\nabla g(x(t)) = 0,
\end{equation}
where for $(x_{1}, x_{2})\in\R^{2}$ we set $g(x_{1}, x_{2}) = \frac{1}{2}(x_{1}^{2} + 100x_{2}^{2})$ and therefore $\nabla g(x_{1}, x_{2}) = (x_{1}, 100x_{2})$. A trajectory generated by $(\ref{system A = 0})$ is a pair $x(t) = (x_{1}(t), x_{2}(t))$. Figure \ref{fig:1} plots both components of the solution to $(\ref{system A = 0})$ with initial Cauchy data $x_{0} = (1, 1)$, $u_{0} = (1, 1)$. Notice that the Lipschitz constant of $\nabla g$ is $L_{\nabla g} = 100$, which means that the cocoercitivity modulus of $\nabla g$ is $\beta = \frac{1}{L_{\nabla g}} = \frac{1}{100}$. To fulfill $\eta > \frac{1}{\beta(\alpha - 1)^{2}} = \frac{100}{(\alpha - 1)^{2}}$, we choose $\alpha = 20$, $\eta = 0.278$. Figure \ref{fig:1a} corresponds to the case with no Hessian damping, that is, $\xi = 0$. Figure \ref{fig:1b} corresponds to a Hessian damping parameter $\xi = 0.2$. 

\begin{figure}[H]
  \begin{subfigure}{0.52\textwidth}
    \includegraphics[width=\linewidth]{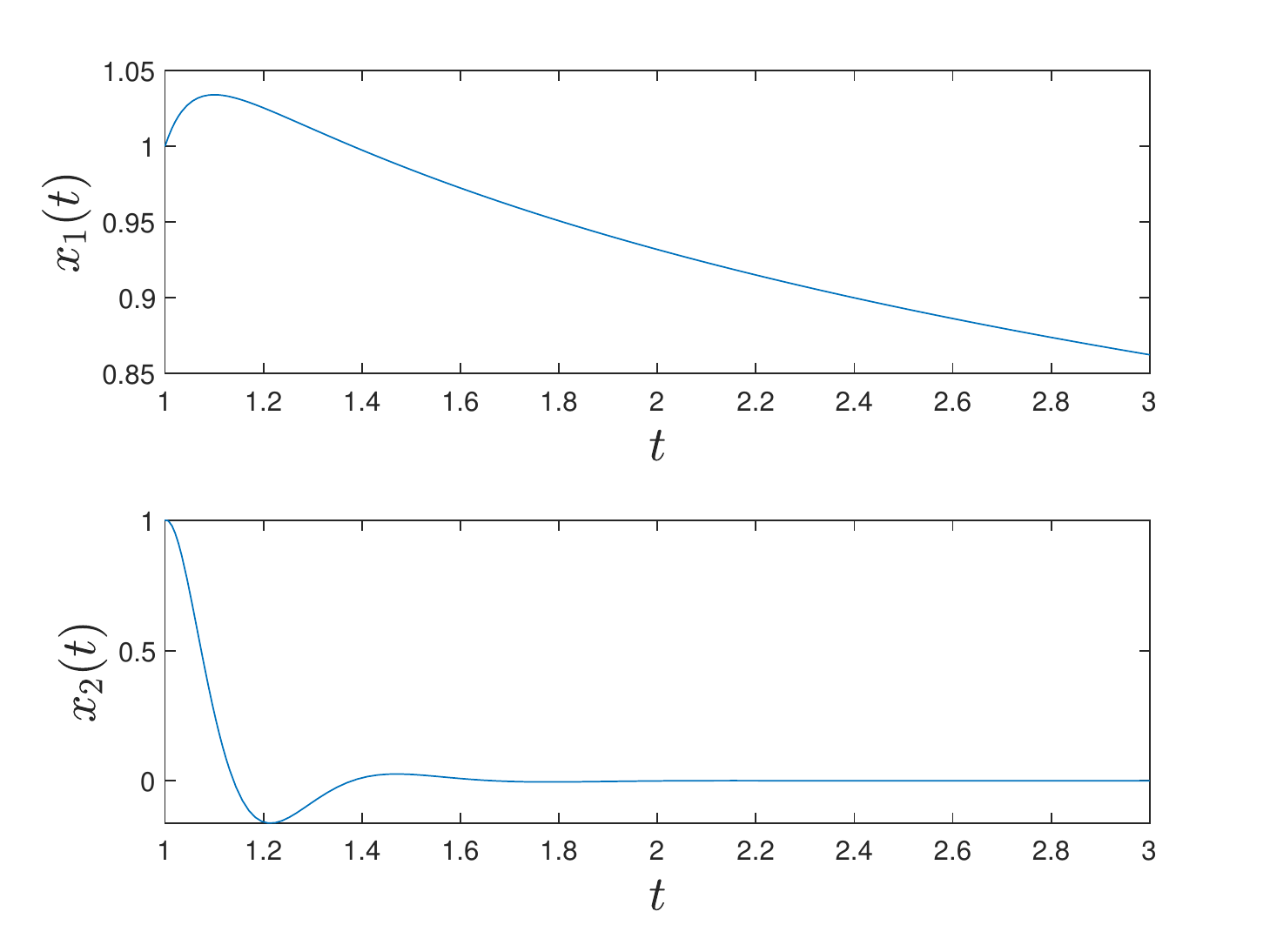}
    \caption{} \label{fig:1a}
  \end{subfigure}
  \hskip -5ex
  \begin{subfigure}{0.52\textwidth}
    \includegraphics[width=\linewidth]{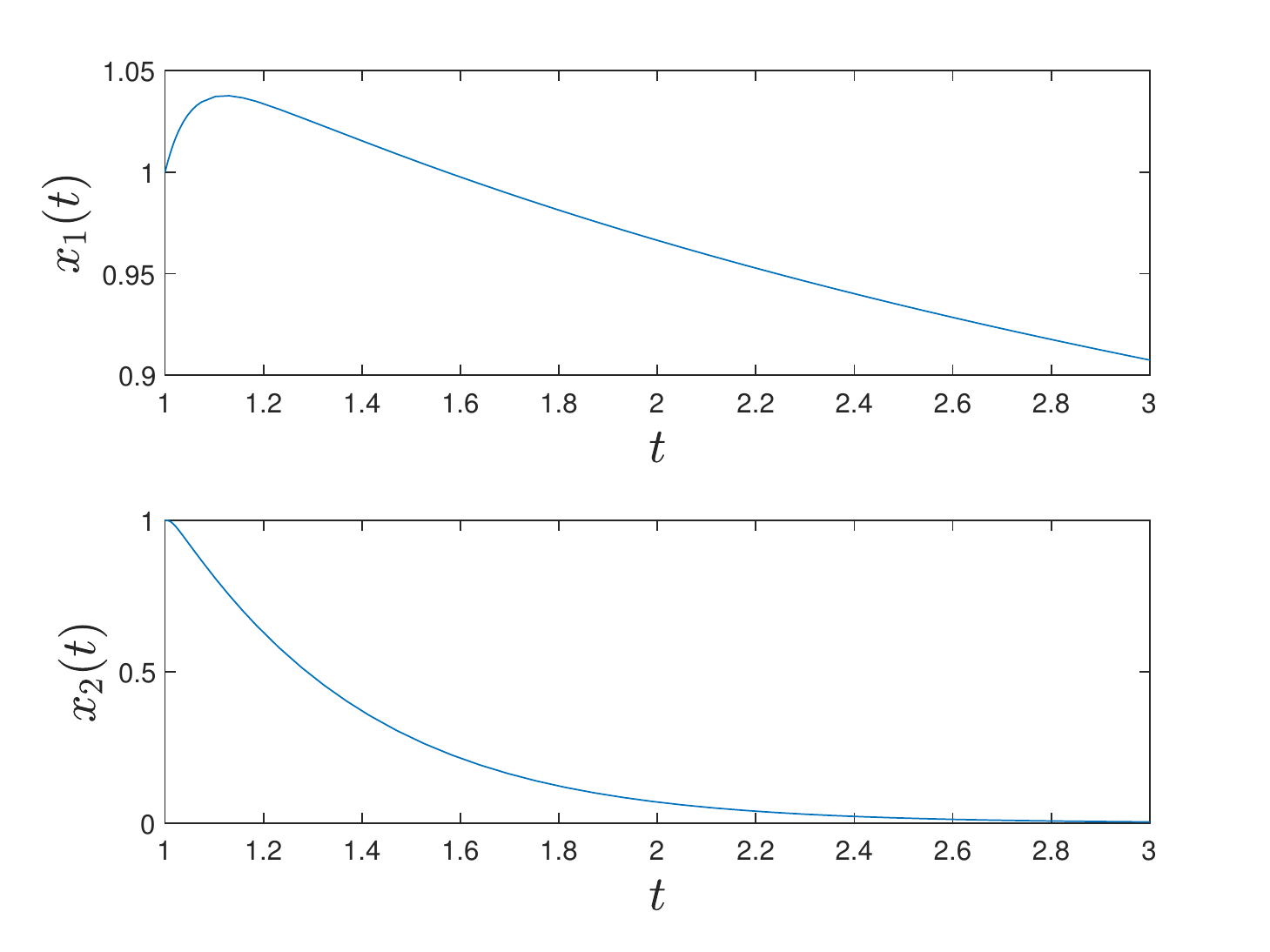}
    \caption{} \label{fig:1b}
  \end{subfigure}
\caption{Trajectories of (Split-DIN-AVD) for $B = \nabla g$} \label{fig:1}
\end{figure}

\noindent Figure \ref{velocities} depicts the fast convergence of the velocities to zero for the cases $\xi = 0$ (Figure \ref{fig:4a}) and $\xi = 0.2$ (Figure \ref{fig:4b}). In both figures, notice the effect of the damping parameter $\xi > 0$, which attenuates the oscillations of the second component of the trajectories, as well as the oscillations present in the velocities.  

\begin{figure}[H]
  \begin{subfigure}{0.52\textwidth}
    \includegraphics[trim = {0 1.2cm 0 2cm}, clip, width=\linewidth]{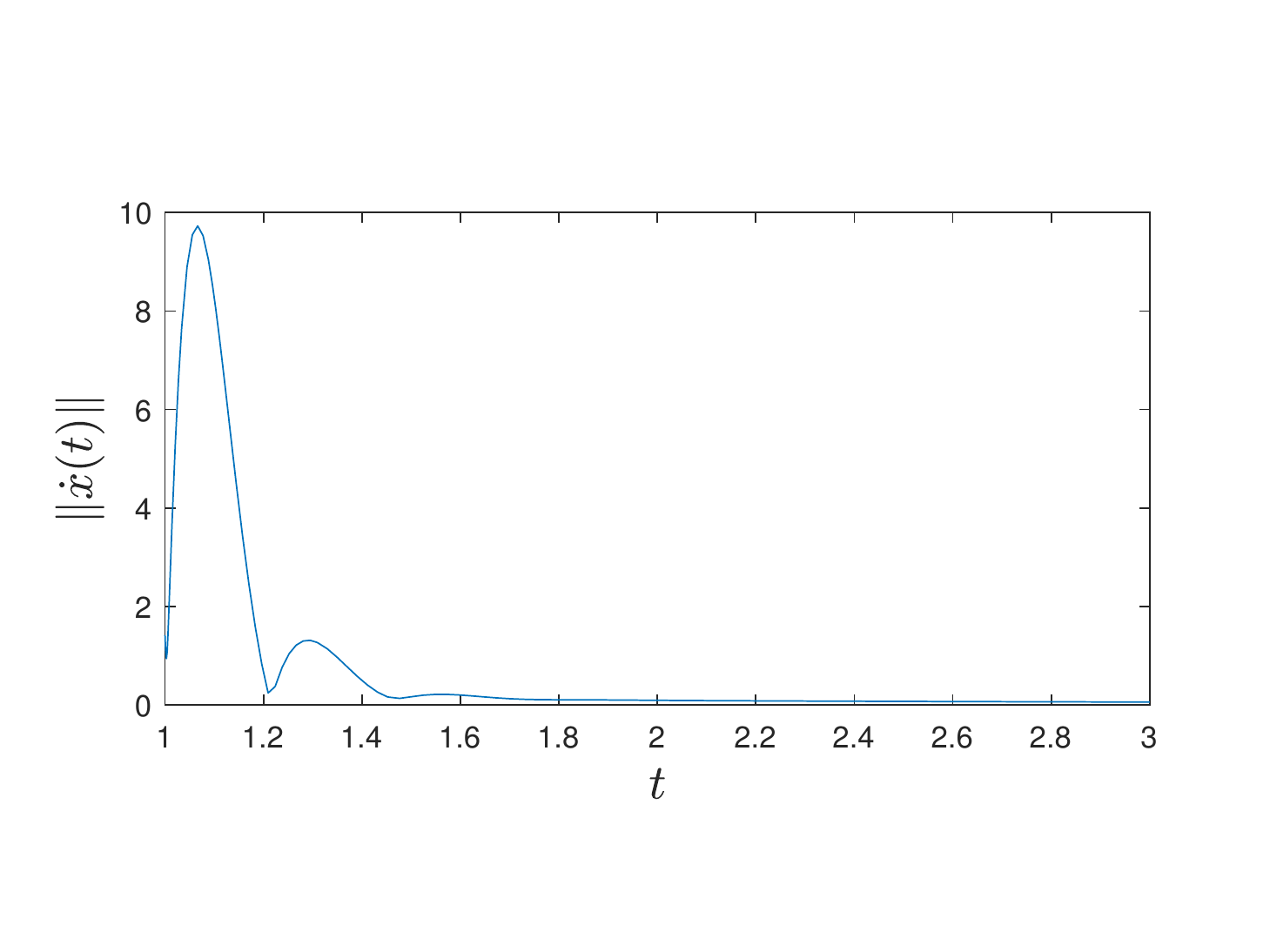}
    \caption{} \label{fig:4a}
  \end{subfigure}
  \hskip -5ex
  \begin{subfigure}{0.52\textwidth}
    \includegraphics[trim = {0 1.2cm 0 2cm}, clip, width=\linewidth]{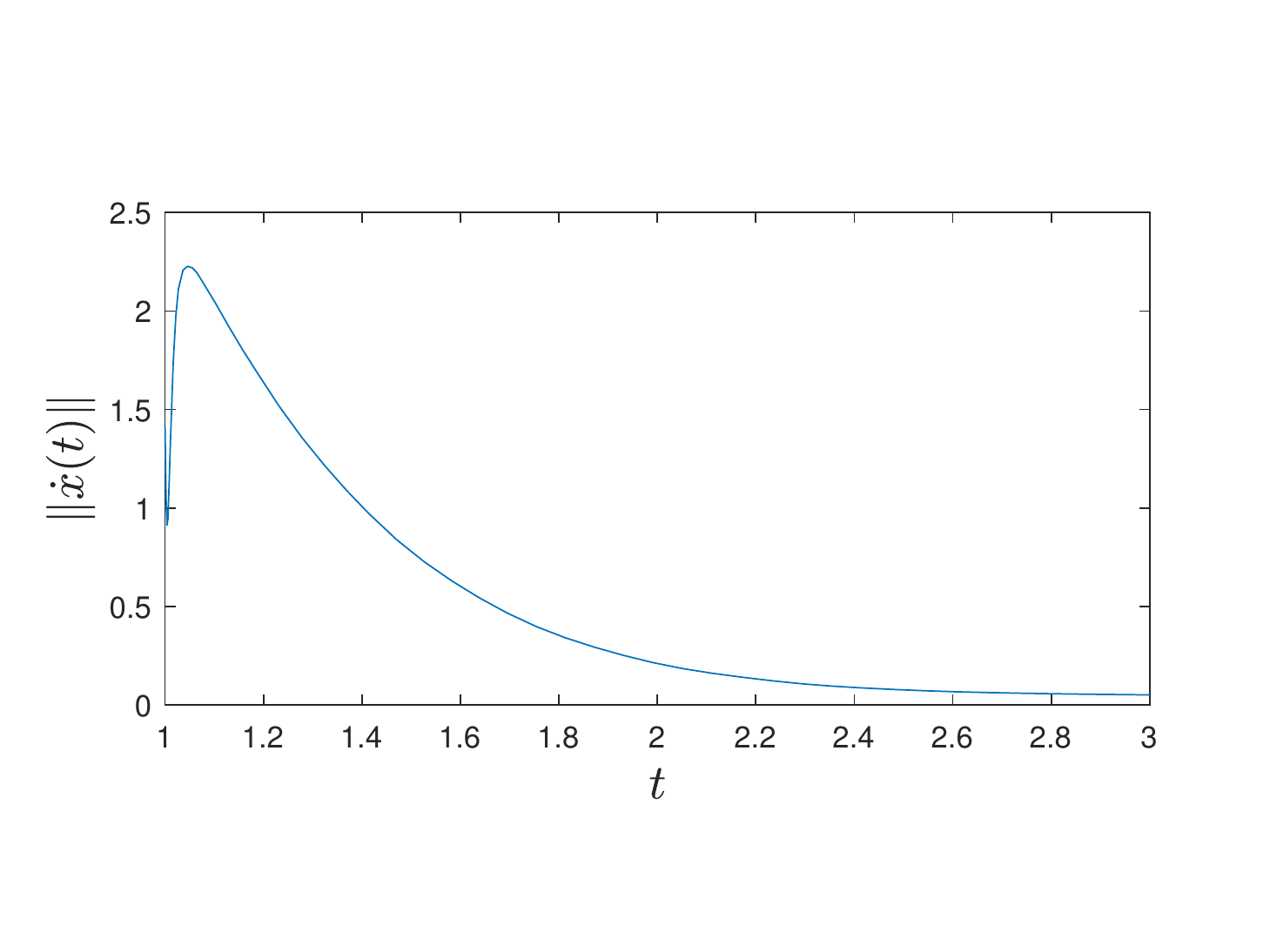}
    \caption{} \label{fig:4b}
  \end{subfigure}
  \caption{Fast convergence of the velocities}
  \label{velocities}
\end{figure}

\subsection{Minimizing a nonsmooth and convex function}

As an example of a continuous time scheme minimizing a proper, convex and lower semicontinuous function $f : \mathcal{H} \rightarrow \R\cup\{+\infty\}$ via (Split-DIN-AVD), we consider the system
\begin{equation}\label{system B = 0}
\ddot{x}(t) + \frac{\alpha}{t} + \xi \frac{d}{dt}\left(\frac{\gamma(t)}{\lambda(t)}\nabla f_{\gamma(t)}(x(t))\right) + \frac{\gamma(t)}{\lambda(t)}\nabla f_{\gamma(t)}(x(t)) = 0.
\end{equation}
We will consider three options for $f$ and plot for each of them the trajectories, the objective function values and the gradients of the Moreau envelopes as follows:
\begin{itemize}
    \item $f(x) = \frac{1}{2}x^{2}$ (Figures \ref{fig:2a} and \ref{fig:5a}),
    \item $f(x) = |x|$ (Figures \ref{fig:2b} and \ref{fig:5b}),
    \item $f(x) = |x| + \frac{1}{2}x^{2}$ (Figures \ref{fig:2c} and \ref{fig:5c}).
\end{itemize}
In order to fulfill $\alpha > 1$ and $\lambda > \frac{1}{(\alpha - 1)^{2}}$, we choose the parameters $\alpha = 2$,  $\lambda = 1.1$, and we take $\xi = 0$ and $\gamma(t) = t^{8}$. We compare the results given by (DIN-AVD) (that is, when $\gamma(t) = \lambda t^{2}$) and the ones given by our system (Split-DIN-AVD). The choice of $\xi$ does not seem to change the plots in a significant way for the examples we have chosen. 
\begin{figure}[H]
  \begin{subfigure}{0.34\textwidth}
    \includegraphics[width=\linewidth]{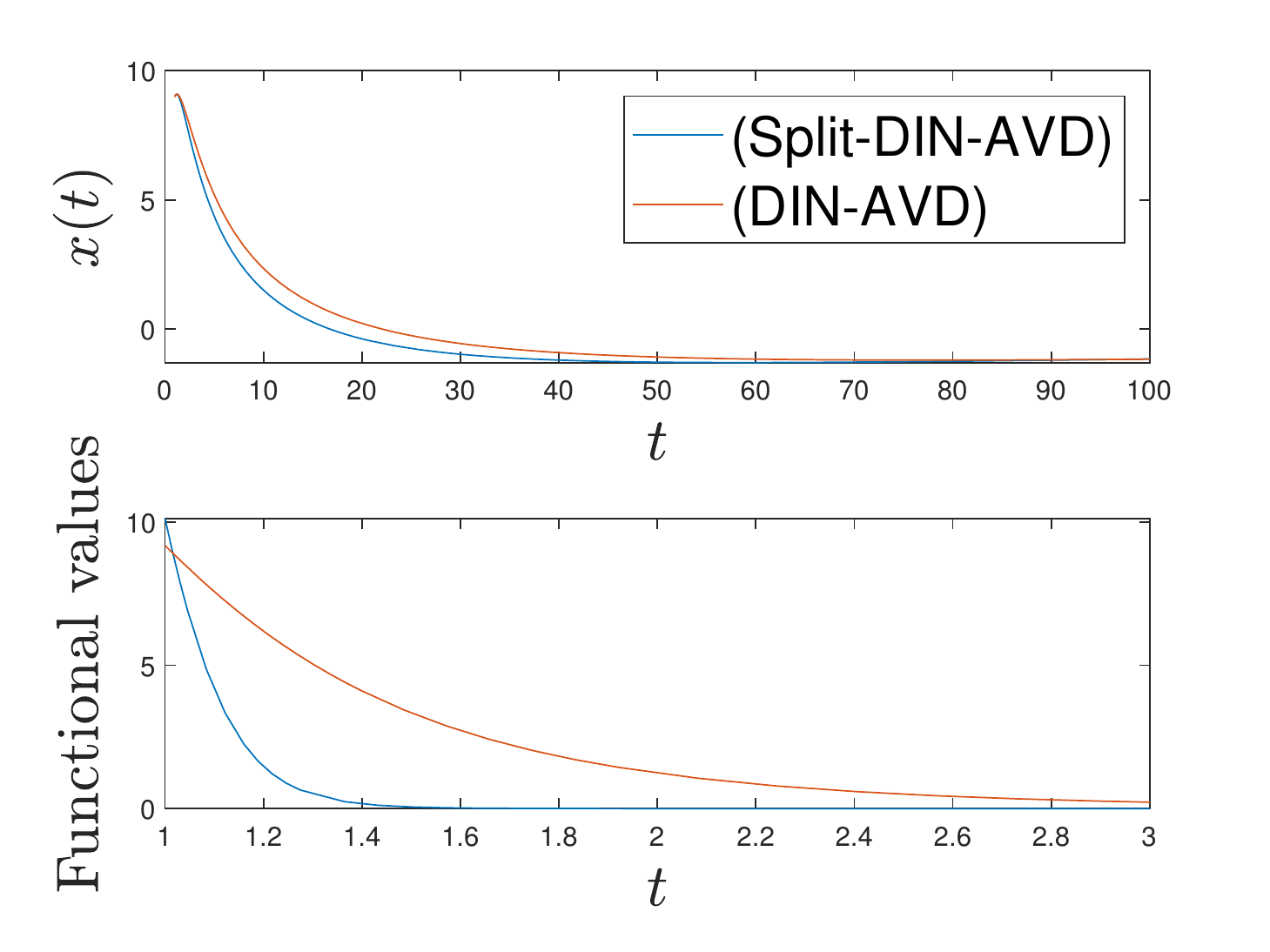}
    \caption{} \label{fig:2a}
  \end{subfigure}%
  \hskip -2ex
  \begin{subfigure}{0.34\textwidth}
    \includegraphics[width=\linewidth]{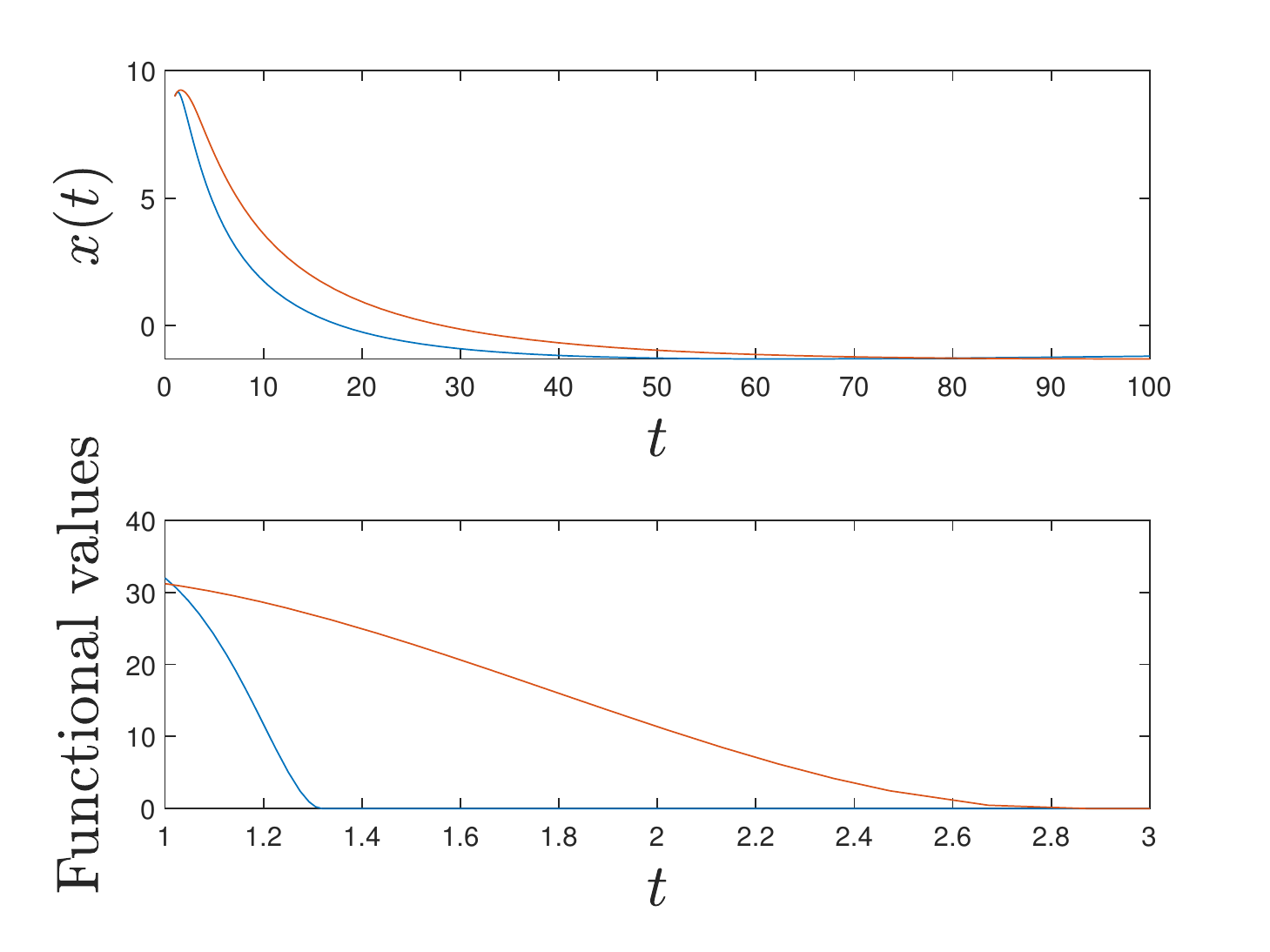}
    \caption{} \label{fig:2b}
  \end{subfigure}
  \hskip -3ex
  \begin{subfigure}{0.34\textwidth}
    \includegraphics[width=\linewidth]{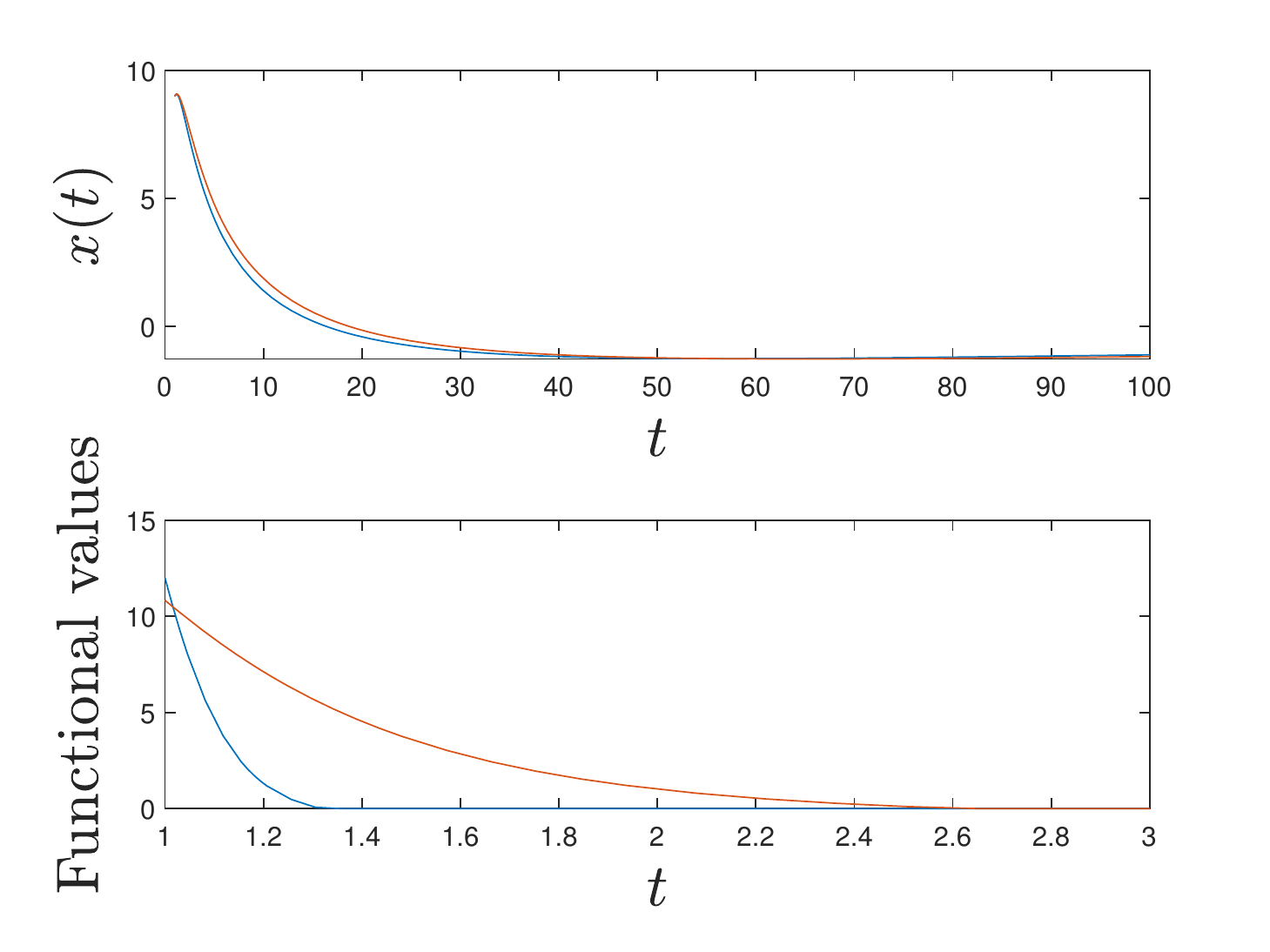}
    \caption{} \label{fig:2c}
  \end{subfigure}
\caption{Trajectories and objective function values in the case $A = \partial f$} \label{fig:2}
\label{functional values}
\end{figure}
\begin{figure}[H]
  \begin{subfigure}{0.34\textwidth}
    \includegraphics[trim = {0 1.2cm 0 2cm}, clip, width=\linewidth]{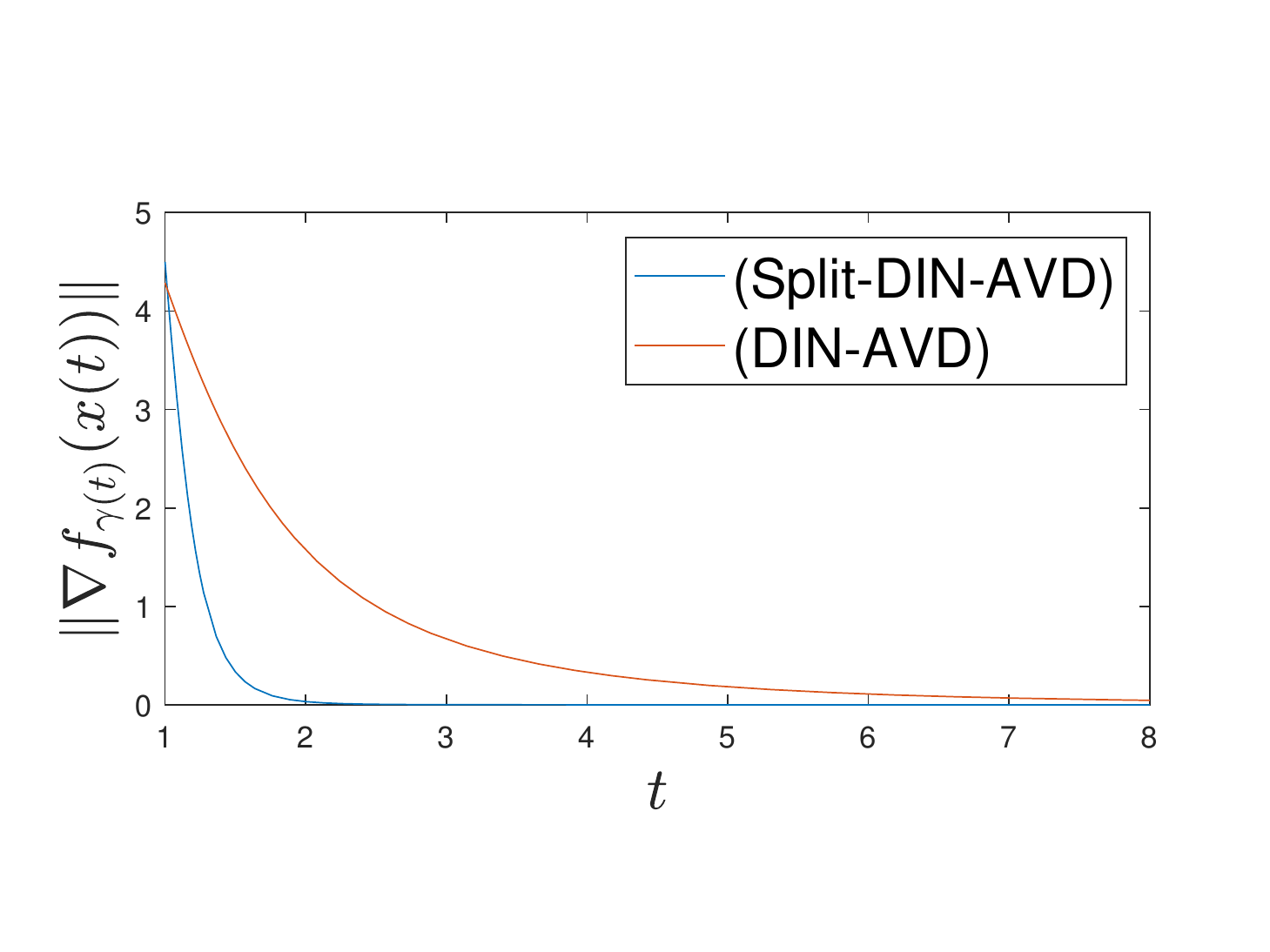}
    \caption{} \label{fig:5a}
  \end{subfigure}%
  \hskip -2ex
  \begin{subfigure}{0.34\textwidth}
    \includegraphics[trim = {0 1.2cm 0 2cm}, clip, width=\linewidth]{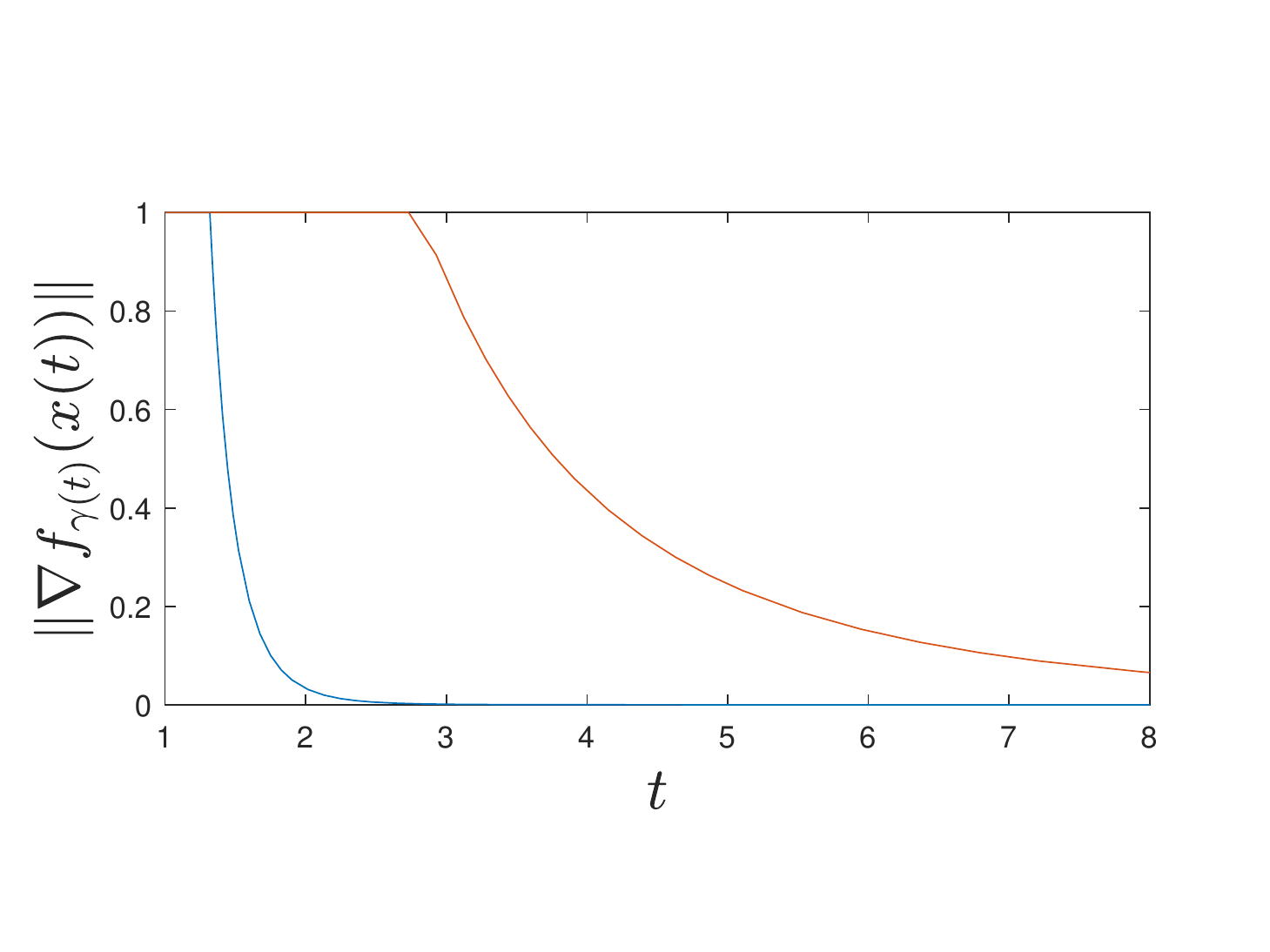}
    \caption{} \label{fig:5b}
  \end{subfigure}
  \hskip -3ex
  \begin{subfigure}{0.34\textwidth}
    \includegraphics[trim = {0 1.2cm 0 2cm}, clip, width=\linewidth]{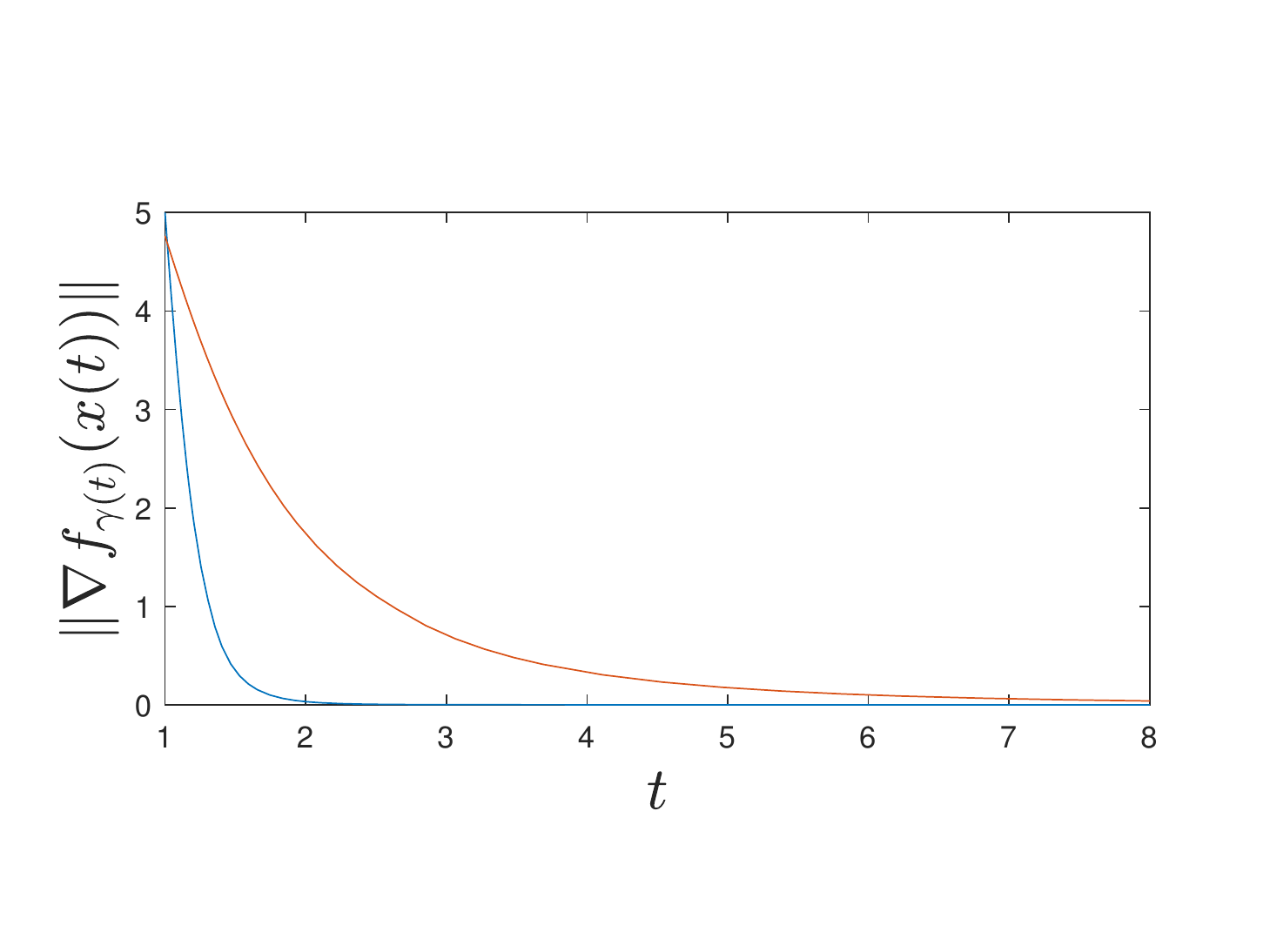}
    \caption{} \label{fig:5c}
  \end{subfigure}
\caption{Gradients of the Moreau envelopes of $f$} 
\label{moreau envelopes}
\end{figure}
Figure \ref{functional values} depicts the trajectories $x(t)$ of $(\ref{system B = 0})$ and the function values $f\left(\prox_{\gamma(t)}(x(t))\right)$ for our choices of $f$ as $t \rightarrow +\infty$. Figure \ref{moreau envelopes} portrays the fast convergence to zero of $\|\nabla f_{\gamma(t)}(x(t))\|$ as $t\to +\infty$. Notice the big improvement over (DIN-AVD) for nonsmooth convex minimization in \cite{AttouchLaszlo} when choosing $\gamma(t) = t^{8}$, a result which we already knew theoretically. Polynomials of high degree seem to be the ones which give the biggest improvements in terms of rates. 

\subsection{An example with operator splitting}

Now we consider the monotone inclusion problem \eqref{monotoneinclusion} for $A(x_{1}, x_{2}) = (-x_{2}, x_{1})$ and $B(x_{1}, x_{2}) = (x_{1}, x_{2})$ for every $(x_{1}, x_{2})\in\R^{2}$. For every $(x_{1}, x_{2}) \in \R^{2}$, an easy calculation gives
\[
J_{\gamma A} 
\begin{bmatrix}
x_{1} \\
x_{2}
\end{bmatrix}
= 
\begin{bmatrix}
\frac{1}{1 + \gamma^{2}} & \frac{\gamma}{1 + \gamma^{2}} \\
\frac{-\gamma}{1 + \gamma^{2}} & \frac{1}{1 + \gamma^{2}}
\end{bmatrix}
\begin{bmatrix}
x_{1} \\
x_{2}
\end{bmatrix},
\]
and so
\begin{align*}
(\id - J_{\gamma A}(\id - \gamma \id))
\begin{bmatrix}
x_{1} \\
x_{2}
\end{bmatrix} &= 
\begin{bmatrix}
x_{1} \\
x_{2}
\end{bmatrix}
- (1 - \gamma)
\begin{bmatrix}
\frac{1}{1 + \gamma^{2}} & \frac{\gamma}{1 + \gamma^{2}} \\
\frac{-\gamma}{1 + \gamma^{2}} & \frac{1}{1 + \gamma^{2}}
\end{bmatrix}
\begin{bmatrix}
x_{1} \\
x_{2}
\end{bmatrix} = \begin{bmatrix}
\frac{\gamma^{2} + \gamma}{1 + \gamma^{2}} & \frac{\gamma - 1}{1 + \gamma^{2}} \\
\frac{1 - \gamma}{1 + \gamma^{2}} & \frac{\gamma^{2} + \gamma}{1 + \gamma^{2}}
\end{bmatrix}
\begin{bmatrix}
x_{1} \\
x_{2}
\end{bmatrix},
\end{align*}
and
\[
T_{\lambda, \gamma}
\begin{bmatrix}
x_{1} \\
x_{2}
\end{bmatrix}
= \begin{bmatrix}
\frac{\gamma^{2} + \gamma}{\lambda(1 + \gamma^{2})} & \frac{\gamma - 1}{\lambda(1 + \gamma^{2})} \\
\frac{1 - \gamma}{\lambda(1 + \gamma^{2})} & \frac{\gamma^{2} + \gamma}{\lambda(1 + \gamma^{2})}
\end{bmatrix}
\begin{bmatrix}
x_{1} \\
x_{2}
\end{bmatrix}.
\]
(Split-DIN-AVD) now reads
\begin{align*}
\begin{bmatrix}
\ddot{x_{1}}(t) \\
\ddot{x_{2}}(t)
\end{bmatrix} +
\frac{\alpha}{t}
\begin{bmatrix}
\dot{x_{1}}(t) \\
\dot{x_{2}}(t)
\end{bmatrix} &+ \:
\xi \frac{d}{dt} \left(
\begin{bmatrix}
\frac{\gamma^{2}(t) + \gamma(t)}{\lambda(t)(1 + \gamma^{2}(t))} & \frac{\gamma(t) - 1}{\lambda(t)(1 + \gamma^{2}(t))} \\
\frac{1 - \gamma(t)}{\lambda(t)(1 + \gamma^{2}(t))} & \frac{\gamma^{2}(t) + \gamma(t)}{\lambda(t)(1 + \gamma^{2}(t))}
\end{bmatrix}
\begin{bmatrix}
x_{1}(t) \\
x_{2}(t)
\end{bmatrix}
\right) \\ 
&+ \begin{bmatrix}
\frac{\gamma^{2}(t) + \gamma(t)}{\lambda(t)(1 + \gamma(t)^{2})} & \frac{\gamma(t) - 1}{\lambda(t)(1 + \gamma^{2}(t))} \\
\frac{1 - \gamma(t)}{\lambda(1 + \gamma^{2}(t))} & \frac{\gamma^{2}(t) + \gamma(t)}{\lambda(t)(1 + \gamma^{2}(t))}
\end{bmatrix}
\begin{bmatrix}
x_{1}(t) \\
x_{2}(t)
\end{bmatrix} = 
\begin{bmatrix}
0 \\
0
\end{bmatrix}.
\end{align*}
We choose the parameters $\alpha = 7$, $\lambda = 0.056$, $\gamma(t) \equiv 1.5$, and the Cauchy data $x_{0} = (1, 2)$ and $u_{0} = (-1, -1)$. Figure \ref{fig:3b} corresponds to the case $\xi = 0$, and Figure \ref{fig:3a} depicts the trajectory when the Hessian damping parameter is $\xi = 0.8$. Again, notice how, not only for optimization problems, but also for monotone inclusions which cannot be reduced to the former, the presence of $\xi$ seems to attenuate the oscillations present in the trajectories. 

\begin{figure}[H]
  \begin{subfigure}{0.52\textwidth}
    \includegraphics[width=\linewidth]{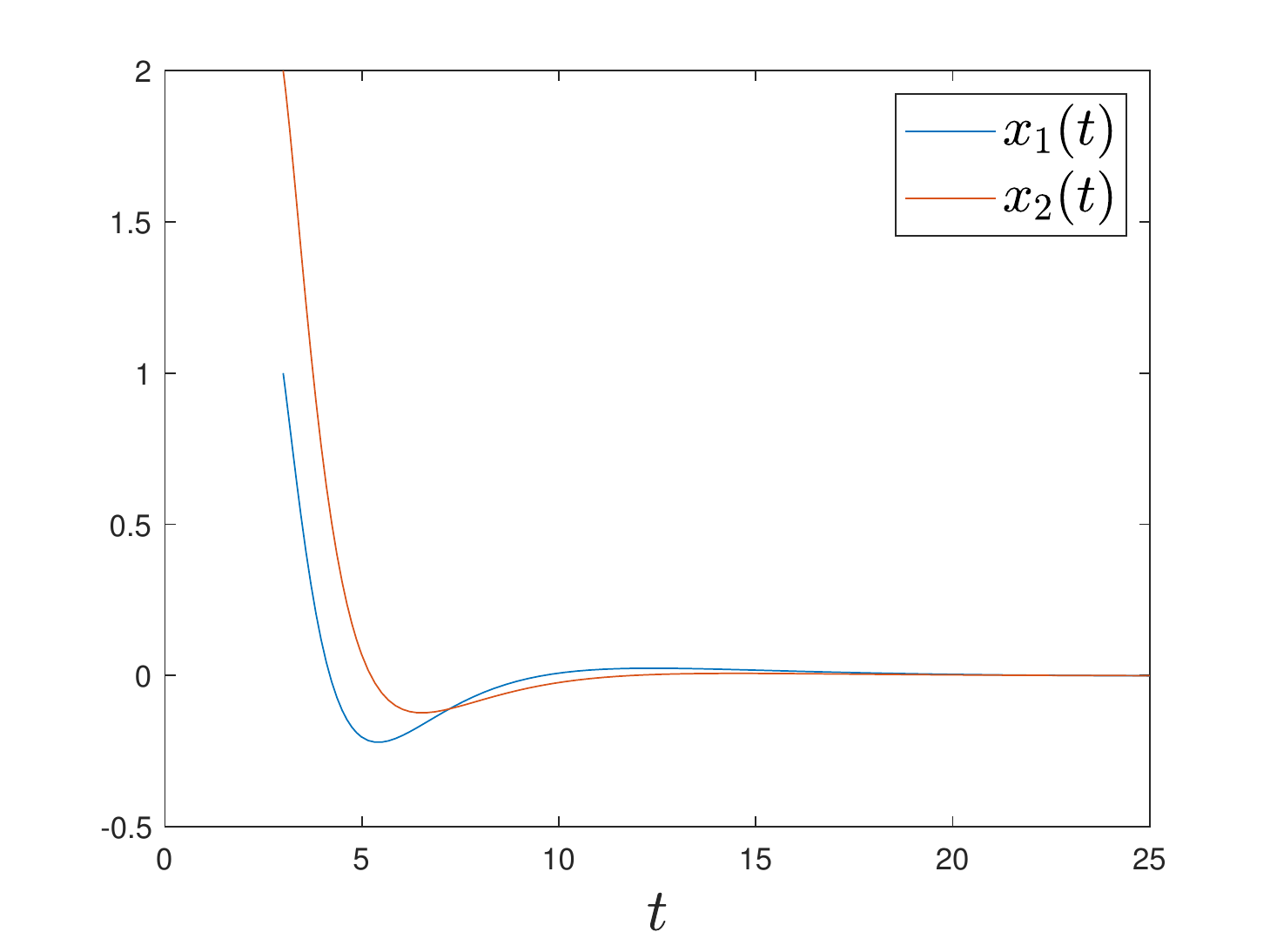}
    \caption{} \label{fig:3b}
  \end{subfigure}
  \hskip -5ex
  \begin{subfigure}{0.52\textwidth}
    \includegraphics[width=\linewidth]{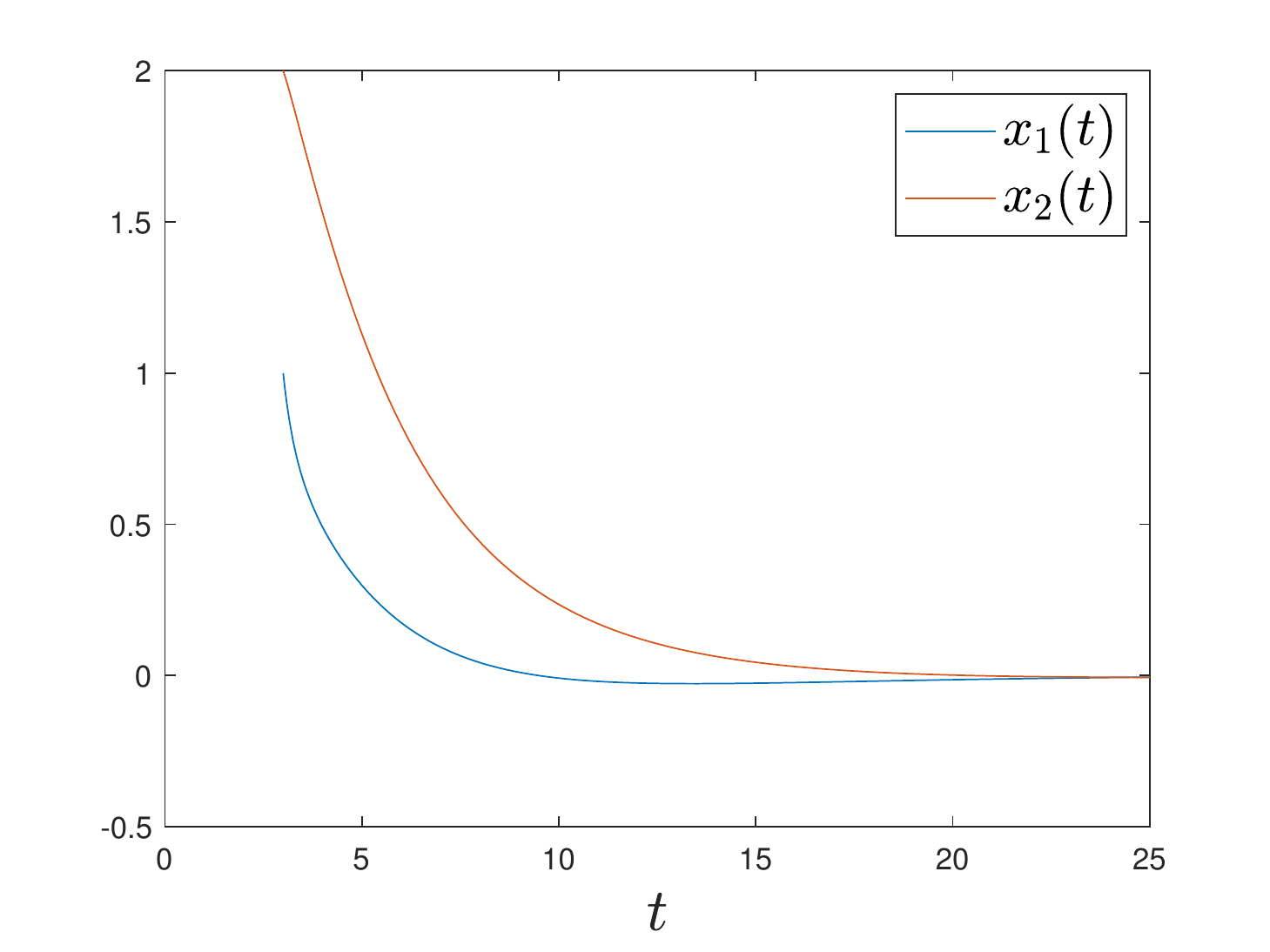}
    \caption{} \label{fig:3a}
  \end{subfigure}

\caption{Trajectories of (Split-DIN-AVD) for finding the zeros of $A + B$} \label{fig:3}
\end{figure}


\section{A numerical algorithm}\label{sec6}

In the following we will derive via time discretization of (Split-DIN-AVD) a numerical algorithm for solving the monotone inclusion problem \eqref{monotoneinclusion}. We perform a discretization of (Split-DIN-AVD) with stepsize $1$ and set, for an integer $k \geq 1$, $x(k) := x_{k}$, $\lambda(k) := \lambda_{k}$, $\gamma(k) := \gamma_{k}$. We make the approximations
\begin{align*}
\ddot{x}(t) &\approx x_{k + 1} - 2x_{k} + x_{k - 1}, &\quad \frac{\alpha}{t}\dot{x}(t) &\approx \frac{\alpha}{k}(x_{k} - x_{k - 1}), \\
\frac{d}{dt}T_{\lambda(t), \gamma(t)}(x(t)) &\approx T_{\lambda_{k}, \gamma_{k}}(x_{k}) - T_{\lambda_{k - 1}, \gamma_{k - 1}}(x_{k - 1}), &\quad T_{\lambda(t), \gamma(t)}(x(t)) &\approx T_{\lambda_{k + 1}, \gamma_{k + 1}}(x_{k + 1}),
\end{align*}
so we get, for every $k\geq 1$, 
\begin{equation}\label{discretesystem}
    x_{k + 1} - 2x_{k} - x_{k - 1} + \frac{\alpha}{k}(x_{k} - x_{k - 1}) + \xi\left(T_{\lambda_{k}, \gamma_{k}}(x_{k}) - T_{\lambda_{k - 1}, \gamma_{k - 1}}(x_{k - 1})\right) + T_{\lambda_{k + 1}, \gamma_{k + 1}}(x_{k + 1}) = 0.
\end{equation}
After rearranging the terms of $(\ref{discretesystem})$, for every $k\geq 1$ we obtain
\begin{equation}\label{discretesystem2}
    x_{k + 1} + T_{\lambda_{k + 1}, \gamma_{k + 1}}(x_{k + 1}) = x_{k} + \left(1 - \frac{\alpha}{k}\right)(x_{k} - x_{k - 1}) - \xi\left(T_{\lambda_{k}, \gamma_{k}}(x_{k}) - T_{\lambda_{k - 1}, \gamma_{k - 1}}(x_{k - 1})\right).
\end{equation}
In other words, after setting $\alpha_{k} = 1 - \frac{\alpha}{k}$ and denoting the right hand side of $(\ref{discretesystem2})$ by $y_{k}$ for every $k\geq 1$, we obtain the following iterative scheme
\begin{equation}
(\forall k \geq 1) \ \left\{
\begin{aligned}\label{algorithm}
    y_{k} &= x_{k} + \alpha_{k}(x_{k} - x_{k - 1}) - \xi\left(T_{\lambda_{k}, \gamma_{k}}(x_{k}) - T_{\lambda_{k - 1}, \gamma_{k - 1}}(x_{k - 1})\right), \\
    x_{k + 1} &= \left(\id + T_{\lambda_{k + 1}, \gamma_{k + 1}}\right)^{-1}(y_{k}).
\end{aligned}
\right.
\end{equation}
Observe that the second step in \eqref{algorithm} is always well-defined. Indeed, for $\lambda, \gamma > 0$, $T_{\lambda, \gamma}$ is $\frac{\lambda}{2}$-cocoercive, hence monotone (see Lemma \ref{lemma1}(i)). This also implies that $T_{\lambda, \gamma}$ is $\frac{2}{\lambda}$-Lipschitz continuous, and a monotone and continuous operator is maximally monotone, according to \cite[Corollary 20.28]{BC}. Hence, by Minty's Theorem (see \cite[Theorem 21.1]{BC}), we know that $\id + T_{\lambda, \gamma} : \mathcal{H}\to \mathcal{H}$ is surjective. 

We are in conditions of stating the main theorem concerning our previous algorithm. 
\begin{theorem}\label{discrete theorem}
Let $A : \mathcal{H}\to 2^{\mathcal{H}}$ be a maximally monotone operator and $B : \mathcal{H}\to \mathcal{H}$ a $\beta$-cocoercive operator for some $\beta \geq 0$ such that $\zer(A + B)\neq \emptyset$. Choose $x_{0}, x_{1}\in\mathcal{H}$ any initial points. Let $\alpha > 1$, $\xi\geq 0$, and $(\lambda_{k})_{k \geq 0}$, $(\gamma_{k})_{k \geq 0}$ sequences of positive numbers that fulfill
\begin{gather*}
    \lambda_{k} = \lambda k^{2} \:\:\: \forall k\geq 1, \quad\text{with}\quad \lambda > \frac{4\xi + 2}{(\alpha - 1)^{2}}, \\
    0 < \inf_{k \geq 0}\gamma_{k} \leq \sup_{k \geq 0}\gamma_{k} < 2\beta \quad \text{and} \quad \frac{\gamma_{k} - \gamma_{k - 1}}{\gamma_{k}} = \mathcal{O}\left(\frac{1}{k}\right) \quad \text{as} \quad k\to +\infty.
\end{gather*}
Now, consider the sequences $(y_{k})_{k\geq 1}$ and $(x_{k})_{k\geq 0}$ generated by algorithm $(\ref{algorithm})$. The following properties are satisfied:
\begin{enumerate}[(i)]
    \item We have the estimates 
    \[
    \|x_{k + 1} - x_{k}\| = \mathcal{O}\left(\frac{1}{k}\right) \quad \text{and} \quad \left\|A_{\gamma_{k}}(x_{k} - \gamma_{k}Bx_{k}) + Bx_{k}\right\| = o\left(\frac{1}{\gamma_{k}}\right) \quad \text{as} \quad k\to +\infty.
    \]
    \item The sequence $(x_{k})_{k\geq 0}$ converges weakly to an element of $\zer(A + B)$.
    \item The sequence $(y_{k})_{k\geq 1}$ converges weakly to an element of $\zer(A + B)$. Precisely, we have $\|x_{k} - y_{k}\| = \mathcal{O}\left(\frac{1}{k}\right)$ as $k\to +\infty$.
\end{enumerate}
\end{theorem}
The proof can be done by transposing the techniques used in the continuous time case to the discrete time case. Algorithm $(\ref{algorithm})$ can be seen as a splitting version of the (PRINAM) algorithm studied by Attouch and L\'aszl\'o in \cite{AttouchLaszloDiscrete}.

\begin{remark}
The second step in $(\ref{algorithm})$ can be quite complicated to compute. However, if $B = 0$, we can resort to the fact that $(A_{\lambda_{1}})_{\lambda_{2}} = A_{\lambda_{1} + \lambda_{2}}$ for $\lambda_{1}, \lambda_{2} > 0$. We now have, for $\lambda, \gamma > 0$,
\[
T_{\lambda, \gamma} = \frac{1}{\lambda}\Big[\id - J_{\gamma A}\Big] = \frac{\gamma}{\lambda}A_{\gamma},
\]
which gives
\begin{equation*}
    \left(\id + T_{\gamma, \lambda}\right)^{-1} = J_{\frac{\gamma}{\lambda}A_{\gamma}} = -\frac{\gamma}{\lambda}\left(A_{\lambda}\right)_{\frac{\gamma}{\lambda}} + \id = \id - \frac{\gamma}{\lambda}A_{\lambda + \frac{\gamma}{\lambda}}.
\end{equation*}
It is now possible to write $(\ref{algorithm})$ in terms of the resolvents of $A$. We have, for every $k\geq 1$,
\begin{align*}
    T_{\lambda_{k}, \gamma_{k}}(x_{k}) - T_{\lambda_{k - 1}, \gamma_{k - 1}}(x_{k - 1}) = \: &\frac{1}{\lambda_{k}}\Big[x_{k} - J_{\gamma_{k}A}(x_{k})\Big] - \frac{1}{\lambda_{k - 1}}\Big[x_{k - 1} - J_{\gamma_{k - 1}A}(x_{k - 1})\Big] \\
    = \: &\left(\frac{1}{\lambda_{k}} - \frac{1}{\lambda_{k - 1}}\right)x_{k} + \frac{1}{\lambda_{k - 1}}(x_{k} - x_{k - 1}) \\
    &- \left(\frac{1}{\lambda_{k}}J_{\gamma_{k}A}(x_{k}) - \frac{1}{\lambda_{k - 1}}J_{\gamma_{k - 1}A}(x_{k - 1})\right), \\
    y_{k} - \frac{\gamma_{k + 1}}{\lambda_{k + 1}}A_{\lambda_{k + 1} + \frac{\gamma_{k + 1}}{\lambda_{k + 1}}}(y_{k}) = \: &y_{k} - \frac{\gamma_{k + 1}}{\lambda_{k + 1}}\frac{1}{\frac{\lambda_{k + 1}^{2} + \gamma_{k + 1}}{\lambda_{k + 1}}}\left[y_{k} - J_{\left(\lambda_{k + 1} + \frac{\gamma_{k + 1}}{\lambda_{k + 1}}\right)A}(y_{k})\right] \\
    = \: &\frac{\lambda_{k + 1}^{2}}{\lambda_{k + 1}^{2} + \gamma_{k + 1}}y_{k} + \frac{\gamma_{k}}{\lambda_{k + 1}^{2} + \gamma_{k + 1}}J_{\left(\lambda_{k + 1} + \frac{\gamma_{k + 1}}{\lambda_{k + 1}}\right)A}(y_{k}).
\end{align*}
So now $(\ref{algorithm})$ becomes
\begin{equation}
(\forall k \geq 1) \ \left\{
\begin{aligned}\label{algorithm4}
    y_{k} = \: &\left(1 - \xi\left(\frac{1}{\lambda_{k}} - \frac{1}{\lambda_{k - 1}}\right)\right)x_{k} + \left(\alpha_{k} - \frac{\xi}{\lambda_{k - 1}}\right)(x_{k} - x_{k - 1}) \\
    &+ \xi\left(\frac{1}{\lambda_{k}}J_{\gamma_{k}A}(x_{k}) - \frac{1}{\lambda_{k - 1}}J_{\gamma_{k - 1}A}(x_{k - 1})\right), \\
    x_{k + 1} = \: &\frac{\lambda_{k + 1}^{2}}{\lambda_{k + 1}^{2} + \gamma_{k + 1}}y_{k} + \frac{\gamma_{k}}{\lambda_{k + 1}^{2} + \gamma_{k + 1}}J_{\left(\lambda_{k + 1} + \frac{\gamma_{k + 1}}{\lambda_{k + 1}}\right)A}(y_{k}).
\end{aligned}
\right.
\end{equation}
Now, if we assume $0 < \inf_{k\geq 0}\gamma_{k}$ and $\lambda > \frac{2\xi + 1}{(\alpha - 1)^{2}}$ and otherwise keep the hypotheses of Theorem \ref{discrete theorem}, then for the sequences $(x_{k})_{k\geq 0}$ and $(y_{k})_{k\geq 1}$ generated by $(\ref{algorithm4})$, the following statements hold:

\begin{enumerate}[(i)]
    \item We have the estimates 
    \[
    \|x_{k + 1} - x_{k}\| = \mathcal{O}\left(\frac{1}{k}\right) \quad \text{and} \quad \left\|A_{\gamma_{k}}(x_{k})\right\| = o\left(\frac{1}{\gamma_{k}}\right) \quad \text{as} \quad k\to +\infty.
    \]
    \item The sequence $(x_{k})_{k\geq 0}$ converges weakly to an element of $\zer A$.
    \item The sequence $(y_{k})_{k\geq 1}$ converges weakly to an element of $\zer A$ as well. Precisely, we have $\|x_{k} - y_{k}\| = \mathcal{O}\left(\frac{1}{k}\right)$ as $k\to +\infty$.
\end{enumerate}
Notice that the condition required for $(\gamma_{k})_{k\geq 0}$ is fulfilled in particular for $\gamma_{k} = k^{n}$ for every $k\geq 1$ and a natural number $n \geq 1$. Thus, by choosing large $n$, we obtain a fast convergence rate for $A_{\gamma_{k}}(x_{k})$ as $k\to +\infty$.

\end{remark}

\appendix

\section{Appendix}
The following are three auxiliary lemmas that are used in the proof of Theorem \ref{maintheorem}. The proof for Lemma \ref{A2} can be found in \cite{AttouchPeypouquet}, while the proof of Lemma \ref{A3} is straightforward. For the proof of Opial's Lemma, we refer the reader to \cite[Lemma 1.10]{AbbasAttouch}.

\settheoremtag{A.1}
\begin{lemma}\label{A2}
Let $t_{0} > 0$, and let $u: [t_{0}, +\infty)\to \R$ be a continuously differentiable function which is bounded from below. Given $\alpha >1$, a nonnegative function $\theta : [t_{0}, +\infty)\to \R$ and a nonnegative function $k\in L^{1}([t_{0}, +\infty), \R)$, let us assume that 
\[
t\ddot{u}(t) + \alpha\dot{u}(t) + \theta(t) \leq k(t)
\]
for almost every $t \geq t_{0}$. Then, the positive part $[\dot{u}]_{+}$ of $\dot{u}$ belongs to $L^{1}([t_{0}, +\infty), \R)$ and $\lim_{t\to +\infty}u(t)$ exists. Moreover, we have $\int_{t_{0}}^{+\infty}\theta(t)dt < +\infty$.
\end{lemma}
\settheoremtag{A.2}
\begin{lemma}\label{A3}
Let $A, B, C\in \R$ and $\mathcal{H}$ a real Hilbert space. Then the inequality
\[
A\|X\|^{2} + 2C\langle X, Y\rangle + B\|Y\|^{2} \leq 0 
\]
holds for every $X, Y\in\mathcal{H}$ if and only if $A, B\leq 0 $ and $C^{2} - AB \leq 0$.
\end{lemma}

\settheoremtag{A.3}
\begin{lemma}[Opial's Lemma]\label{opial lemma}
Let $S\subseteq \mathcal{H}$ be a nonempty set and $x : [t_{0}, +\infty) \to \mathcal{H}$ a given map, where $t_{0} > 0$. Assume that
\begin{enumerate}[(i)]
    \item for every $x^{*} \in S$, $\lim_{t\to +\infty}\|x(t) - x^{*}\|$ exists;
    \item every weak sequential cluster point of the map $x$ belongs to $S$.
\end{enumerate}
Then, there exists $x_{\infty} \in S$ such that $x(t)$ converges weakly to $x_{\infty}$ as $t\to +\infty$.
\end{lemma}

 
\end{document}